\documentclass[amssymb,reqno,amsfonts,refcheck,12pt,verbatim,righttag]{amsart}

\usepackage{graphicx}
\usepackage{graphics,color}
\usepackage{amssymb,mathrsfs}
\usepackage{graphicx,color,tikz,caption,subcaption}
\usepackage[all]{xy}
\usepackage{mathtools}
\usepackage{enumerate}
\usepackage{amsmath}
\usepackage{bbm}
\usepackage[colorlinks, linkcolor=blue,citecolor=blue, anchorcolor=blue,urlcolor=blue,pagebackref,hypertexnames=false]{hyperref}

\numberwithin{equation}{section} 
\newtheorem{thm}{Theorem}[section]
\newtheorem{lem}[thm]{Lemma}

\newtheorem{cor}[thm]{Corollary}

\newtheorem{rem}[thm]{Remark}

\usepackage{float}
\begin{document}
\baselineskip 14pt
\title{Uniform Diophantine approximation and run-length function in continued fractions}
\author{Bo Tan and Qing-Long Zhou$^{\dag}$}
\address{$^{1}$ School  of  Mathematics  and  Statistics,
                Huazhong  University  of Science  and  Technology, 430074 Wuhan, PR China}
          \email{tanbo@hust.edu.cn}

\address{$^{2}$ School of Science, Wuhan University of Technology, 430070 Wuhan, PR China }
\email{zhouql@whut.edu.cn}
\thanks{$^{\dag}$ Corresponding author.}


\date{}

\begin{abstract}
We study the multifractal properties of the uniform approximation exponent and asymptotic approximation exponent in continued fractions. As a corollary, 
we calculate the Hausdorff dimension of the uniform Diophantine set
$$\mathcal{U}(y,\hat{\nu})=\Big\{x\in[0,1)\colon \forall N\gg1, \exists~ n\in[1,N], \text{ such that }  |T^{n}(x)-y|<|I_{N}(y)|^{\hat{\nu}}\Big\}$$
for algebraic irrational points $y\in[0,1)$.
These results contribute to the study of the uniform Diophantine approximation,
and  apply to investigating the multifractal properties of run-length function in continued fractions.
\end{abstract}

\maketitle
{\small \tableofcontents}

\section{Introduction}

\subsection{Uniform Diophantine approximation}
The classical metric Diophantine approximation is concerned with the question of how well a real number can be approximated by rationals.
A qualitative answer is provided by the fact that the set of rationals is dense in the reals.
Dirichlet pioneered the quantitative study by showing that, for any $x\in \mathbb{R}$ and $Q>1,$
there exists $(p,q)\in\mathbb{Z}\times\mathbb{N}$ such that
\begin{equation}\label{D1}
|qx-p|\le \frac{1}{Q}~~\text{and } q<Q.
\end{equation}
The result serves as a start point of the metric theory in Diophantine approximation.
An easy application yields the following corollary$\colon$ for any $x\in \mathbb{R},$
there exists infinitely many $(p,q)\in\mathbb{Z}\times\mathbb{N}$ such that
\begin{equation*}
|qx-p|\le \frac{1}{q}.
\end{equation*}
This corollary claims that $|qx-p|$ is small compared to $q$, while Dirichlet's original theorem (\ref{D1}) provides  a uniform estimate of $|qx-p|$  in terms of $Q$. These two kinds of approximations are referred to as uniform approximation and asymptotic respectively. See \cite{W12} for more account on the related subject.

In this article we are interested in the numbers which are approached in uniform or asymptotic way by an orbit (in a dynamical system) with a prescribed speed.
Let $(X,  T, \mu)$ be a measure-preserving dynamical system, where
 $(X,d)$ is a metric space,  $T\colon X\to X$ is a Borel transformation, and   $\mu$ is a $T$-invariant Borel probability measure on $X$.
As is well known, Birkhoff's ergodic theorem \cite{PW} implies that,  in an ergodic dynamical system,  for almost all $y\in X$, the set
$$\left\{x\in X\colon \liminf_{n\to\infty}d(T^{n}(x),y)=0\right\}$$
is of full $\mu$-measure.
The  result, which gives a qualitative characterization of the distributions of the $T$-orbits in $X$, can be regarded as a counterpart of  the density property of rational numbers in the reals. It  leads naturally to the   quantitative study of the distributions of the $T$-orbits.

The shrinking target problem  in dynamical system $(X,T)$ aims at a quantitative study on the Birkhoff's ergodic theorem, which investigates the set
$$W_y(T,\psi)=\Big\{x\in X\colon d(T^{n}(x),y)<\psi(n) \text{ for infinitely many } n\in \mathbb{N}\Big\},$$
where $\psi\colon \mathbb{N}\to \mathbb{R}$ is a positive function such that $\psi(n)\to 0$ as $n\to\infty,$ and $y\in X.$
Hill $\&$ Velani \cite{HV95} studied the Hausdorff dimension of the set
$$\Big\{x\in X\colon d(T^{n}(x),y)<e^{-\tau n} \text{ for infinitely many } n\in \mathbb{N}\Big\}$$
in the system $(X,T)$ with $T$ an expanding rational map of degree greater than or equal to 2 and $X$ the corresponding Julia set, where $\tau>0.$ See \cite{TW} for more information.

Representations of real numbers are often induced by dynamical systems or algorithms,
and thus the related Diophantine approximation problems are in the nature of dynamical system, fractal geometry and number theory. An active topic of research lies in studying the approximation of real numbers in dynamical systems by the    orbits of the points. Recently, many researchers have studied the Hausdorff dimension of the set $W_y(T,\psi)$ in the corresponding dynamical system under different expansions, and obtained many significant results \cite{LWWX,SW,STZ,TZ}.
Marked by the famous mass transfer principle established by Beresnevich $\&$ Velani \cite{BV}, studies on the asymptotic approximation properties of orbits in dynamical systems are relatively mature.
However, there are   few  results on the uniform approximation properties of orbits.

Let $\big(X,T\big)$ be a exponentially mixing system with respect to the probability measure $\mu$, and let $\psi\colon \mathbb{N}\to \mathbb{R}$ be a positive function  satisfying that $\psi(n)\to 0$ as $n\to\infty$.
Kleinbock, Konstantous $\&$ Richter \cite{KKR} studied the Lebesgue measure of the set of real numbers $x\in X$ with the property that,
for every sufficiently large integer $N,$ there is an integer $n$ with $1\le n\le N$ such that the distance between $T^{n}(x)$ and a fixed $y$ is at most   $\psi(N),$ i.e.,
$$\mathcal{U}(y,\psi)=\Big\{x\in[0,1)\colon \forall N\gg1, \exists~n\in[1,N], \text{ such that }  |T^{n}(x)-y|<\psi(N)\Big\}.$$
They gave the sufficient conditions for $\mathcal{U}(y,\psi)$ to be of zero or full measure.
Although the Khintchine type 0-1 law of the set $\mathcal{U}(y,\psi)$ has not been established,
the work has aroused the interest of researchers  (see \cite{GP,KA,KKP} for the related studies).
Bugeaud $\&$ Liao \cite{BL} investigated the size of the set
$$\Big\{x\in[0,1)\colon \forall N\gg1, \exists~ n\in[1,N], \text{ such that }  T_\beta^{n}(x)<|I_{N}(0)|^{\hat{\nu}}\Big\}$$
in $\beta$-dynamical systems from the perspective of Hausdorff dimension, where $T_\beta$ is the $\beta$-transformation on $[0,1)$ defined by $T_\beta(x)=\beta x \text{ mod }1,$ $I_n(0)$ denotes the basic interval of order $n$ which contains the point 0, and $\hat{\nu}$ is a nonnegative real number. For more information related to the uniform approximation properties, see \cite{KL,ZW} and the references therein.

In this paper, we shall investigate the uniform approximation properties of the orbits  under the Gauss transformation.

The Gauss transformation $T\colon [0,1) \to [0,1)$ is defined as  
\begin{equation*}
T(0)=0,~ T(x)= \frac{1}{x}\!\!\!\pmod1~\text{for} ~x\in(0,1).
\end{equation*}
And each irrational number $x\in[0,1)$ can be uniquely expanded into the following form$\colon$
\begin{equation}\label{e1}
x=\frac{1}{a_{1}(x)+\frac{1}{a_{2}(x)+\ddots+\frac{1}{a_{n}+T^{n}(x)}}}
 =\frac{1}{a_{1}(x)+\frac{1}{a_{2}(x)+\frac{1}{a_{3}(x)+\ddots}}},
\end{equation}
with $a_{n}(x)=\lfloor\frac{1}{T^{n-1}(x)}\rfloor$, called the $n$-th partial quotient of $x$ (here $\lfloor\cdot\rfloor$ denotes the greatest integer less than or equal to a real number and $T^0$ denotes the identity map).
For simplicity of notation, we write $(\ref{e1})$ as
\begin{equation}\label{e2}
x=[a_{1}(x),a_{2}(x),\ldots,a_{n}(x)+T^{n}(x)]=[a_{1}(x),a_{2}(x),a_{3}(x),\ldots].
\end{equation}
As was shown by Philipp \cite{P67}, the system $([0,1),T)$ is exponentially mixing with respect to the Gauss measure $\mu$ given by $d\mu=dx/(1+x)\log 2.$
Thus  the above result of \cite{KKR} applies for the Gauss measure of the set $\mathcal{U}(y,\psi)$ in the system of continued fractions.
In consequence, we shall focus on the size of $\mathcal{U}(y,\psi)$ in dimension.

The dimension of  sets $\mathcal{U}(y,\psi)$   depend on the choice of the given point $y$. In this paper, we will consider a class of quadratic irrational numbers $y={(\sqrt{i^{2}+4}-i)}/{2}=[i,i,\ldots]$  with $i\in \mathbb{N}$,
and calculate the Hausdorff dimension of the set
$$\mathcal{U}(y,\hat{\nu})=\Big\{x\in[0,1)\colon \forall N\gg1, \exists~ n\in[1,N], \text{ such that }  |T^{n}(x)-y|<|I_{N}(y)|^{\hat{\nu}}\Big\}.$$
For $\beta \in[0,1]$, let $s(\beta,{y})$ denote the solution of
$$P\Big(T,-s\Big(\log |T'|+\frac{\beta}{1-\beta}{g(y)}\Big)\Big)=0,$$
where $P(T,\phi)$ is the pressure function with potential $\phi$ in the continued fraction system $([0,1),T)$, $T'$ is the derivative of $T$, {and $g(y)$ is the limit $\lim_{n}\log q_{n}(y)/n$.}
\begin{thm}\label{TZ1}
Given a nonnegative real number $\hat{\nu},$ we have
\begin{equation*}
 \dim_{H}\mathcal{U}(y,\hat{\nu})=\left\{
    \begin{array}{ll}
      s\Big(\frac{4\hat{\nu}}{(1+\hat{\nu})^{2}},{y}\Big), & \ \ \ \text{if }~ 0\leq \hat{\nu}\leq 1; \\
      0, & \ \ \ \text{otherwise.}
    \end{array}
  \right.
\end{equation*}
Throughout the paper, $\dim_{H}$ denotes the Hausdorff dimension of a set.
\end{thm}

We now turn to the discussion of two approximation exponents which are relevant to asymptotic/uniform Diophantine approximation. For $x\in[0,1),$ 
we define the asymptotic approximation exponent of $x$ by
$$\nu(x)=\sup\Big\{\nu\ge 0\colon |T^{n}(x)-y|<|I_n(y)|^{\nu}\text{ for infinitely many } n\in \mathbb{N}\Big\}$$
and the uniform approximation exponent by
$$\hat{\nu}(x)=\sup\Big\{\hat{\nu}\ge 0\colon \forall N\gg1, \exists~ n\in[1,N], \text{ such that }  |T^{n}(x)-y|<|I_N(y)|^{\hat{\nu}}\Big\}.$$
The exponents $\nu(x)$ and $\hat{\nu}(x)$ are analogous to the exponents introduced in \cite{AB}, see also \cite{BL,BL05}.
By the definitions of $\nu(x)$ and $\hat{\nu}(x)$, it is readily checked that $\hat{\nu}(x)\le \nu(x)$ for all $x\in[0,1).$
Actually, applying Philipp's result \cite{P67}, we deduce that $\nu(x)=0$ for Lebesgue almost all $x\in[0,1)$ (see Lemma \ref{full}). Li, Wang, Wu $\&$ Xu \cite{LWWX} studied the multifractal properties of the asymptotic exponent $\nu(x)$ and showed that for $0\le\nu\le+\infty,$
\begin{equation}\label{Asy}
 \dim_{H}\{x\in[0,1)\colon \nu(x)\ge \nu\}=s\Big(\frac{\nu}{1+\nu},{y}\Big).
\end{equation}
We will denote by $E(\hat{\nu})$ the level set of the uniform approximation exponent:
$$E(\hat{\nu})=\{x\in[0,1)\colon \hat{\nu}(x)=\hat{\nu}\}.$$

\begin{thm}\label{TZ3}
Given a nonnegative real number $\hat{\nu},$ we have
$$\dim_{H}E(\hat{\nu})=\dim_{H}\{x\in[0,1)\colon \hat{\nu}(x)\ge\hat{\nu}\}=\dim_{H}\mathcal{U}(y,\hat{\nu}).$$
\end{thm}

Actually, Theorems \ref{TZ1} and \ref{TZ3} follow from the following more general result which gives the Hausdorff dimension of the set
\begin{equation*}
  E(\hat{\nu},\nu)=\{x\in[0,1)\colon \hat{\nu}(x)=\hat{\nu},~\nu(x)=\nu\}.
\end{equation*}

\begin{thm}\label{TZ2}
    Given two nonnegative real numbers $\hat{\nu}$ and $\nu$ with $\hat\nu\le\nu$, we have
       \begin{equation*}
         \dim_{H}E(\hat{\nu},\nu)
       =\left\{\begin{array}{ll}  1, & \ \ \ \text{if }~ \nu=0;\\
                                  s\Big(\frac{\nu^{2}}{(1+\nu)(\nu-\hat{\nu})},{y}\Big), & \ \ \ \text{if }~ 0\le\hat{\nu}\le\frac{\nu}{1+\nu}<\nu\leq\infty; \\
                                  0, & \ \ \ \text{otherwise.}
       \end{array}
      \right.
     \end{equation*}
    Here,  we take  $\frac{\nu^{2}}{(1+\nu)(\nu-\hat{\nu})}=1$ when $\nu=\infty$.
\end{thm}

Let us make the following remarks regarding Theorems \ref{TZ1}-\ref{TZ2}:
\begin{itemize}
\item These results remain valid for  any quadratic irrational number $y$. Indeed, by Lagrange's theorem, any such $y$ is represented by a periodic continued fraction expansion, i.e., $$y=[a_1,a_2,\ldots,a_{k_0},\overline{a_{k_{0}+1},\ldots,a_{k_{0}+h}}]$$
    for some positive integers $k_{0}$ and $h.$
   A slight change (replacing  the block  $[i]$ by the periodic block
     $[a_{k_{0}+1},\ldots,a_{k_{0}+h}]$) in the proofs actually shows that   Theorems \ref{TZ1}-\ref{TZ2}   still hold for every quadratic irrational number $y\in[0,1).$

     \medskip

\item The fractal sets $\mathcal{U}(y,\hat{\nu}),$ $E(\hat{\nu})$ and $E(\hat{\nu},\nu)$ are not the so-called limsup sets,
and thus we cannot obtain a natural covering to estimate the upper bound of the Hausdorff dimensions of the sets $\mathcal{U}(y,\hat{\nu})$ and $E(\hat{\nu},\nu)$.  To overcome this difficulty, we need a better understanding on the fractal structure of these sets;  the previous work of Bugeaud $\&$ Liao \cite{BL} helps.
\end{itemize}

Combining (\ref{Asy}) and Theorem \ref{TZ2}, we obtain the dimension of  the level set related to the asymptotic exponent $\nu(x)$.

\begin{cor}\label{Cor}
Given a nonnegative real number $\nu,$ we have
$$\dim_{H}\{x\in[0,1)\colon \nu(x)=\nu\}=s\Big(\frac{\nu}{1+\nu},{y}\Big).$$
\end{cor}

\subsection{Run-length function}
Applying the main ideas of the proofs of Theorems $\ref{TZ1}$ and \ref{TZ2}, we characterize the multifractal properties of run-length function in continued fractions.

The run-length function was initially introduced  in a mathematical experiment of cion tossing,
which counts the consecutive occurrences of `heads' in $n$ times trials.
This function has been extensively studied for a long time. 
For  $x\in[0,1],$ let $r_n(x)$ be the dyadic run-length function of $x,$ namely,
the longest run of 0's in the first $n$ digits of the dyadic expansion of $x.$
Erd\"{o}s $\&$ R\'{e}nyi \cite{E} did a pioneer work on the asymptotic behavior of $r_n(x)\colon$ for Lebesgue almost all $x\in[0,1],$
\begin{equation*}
    \lim_{n\to\infty}\frac{r_n(x)}{\log_2 n}=1.
\end{equation*}

Likewise, we define the run-length function in the continued fraction expansion:
for $n\geq 1$,   the $n$-th maximal run-length function of $x$ is defined as
$$R_{n}(x)=\max\big\{l\geq1\colon a_{i+1}(x)=\cdots=a_{i+l}(x) \text{ for some } 0 \leq i\leq n-l\big\}.$$
Wang $\&$ Wu \cite{W} considered the metric properties of $R_{n}(x)$  and proved that
$$\lim_{n\to\infty}\frac{R_{n}(x)}{\log_{\frac{\sqrt{5}+1}{2}}n}=\frac{1}{2}$$
for almost all $x\in [0,1).$ They also studied the following exceptional sets
\begin{equation*}
    F\big(\{\varphi(n)\}_{n=1}^{\infty}\big)=\Big\{x\in[0,1)\colon \lim_{n\to\infty}\frac{R_{n}(x)}{\varphi(n)}=1\Big\},
\end{equation*}
\begin{equation*}
    G\big(\{\varphi(n)\}_{n=1}^{\infty}\big)=\Big\{x\in[0,1)\colon \limsup\limits_{n\to\infty}\frac{R_{n}(x)}{\varphi(n)}=1\Big\},
\end{equation*}
where $\varphi\colon \mathbb{N} \to \mathbb{R}^{+}$ is a non-decreasing function. They showed that$\colon$

{\rm{(1)}} if $\lim_{n\to\infty}\frac{\varphi(n+\varphi(n))}{\varphi(n)}=1,$
then $\dim_{H}F\big(\{\varphi(n)\}_{n=1}^{\infty}\big)=1;$

{\rm{(2)}} if $\liminf_{n\to\infty}\frac{\varphi(n)}{n}=\beta \in[0,1],$
then $\dim_{H}G\big(\{\varphi(n)\}_{n=1}^{\infty}\big)=s\big(\beta,{\frac{\sqrt{5}+1}{2}}\big)$.

\noindent In the study of Case (2), Wang $\&$ Wu studied essentially the Hausdorff dimension of the following set
\begin{equation}\label{equ13}
G(\beta)=\left\{x\in[0,1)\colon \limsup_{n \to\infty}\frac{R_{n}(x)}{n}=\beta\right\}.
\end{equation}
 Replacing the limsup of the quantity $R_n(x)/n$ in (\ref{equ13}) with liminf, we study the set
$$F(\alpha)=\left\{x\in[0,1)\colon \liminf_{n\to\infty}\frac{R_{n}(x)}{n}=\alpha\right\},$$
and determine the Hausdorff dimension of the intersections of $F(\alpha)\cap G(\beta).$ As a corollary, we  obtain the Hausdorff dimension of $F(\alpha).$
\begin{thm}\label{thm1}
For $\alpha, \beta\in [0,1]$ with $\alpha\leq \beta$, we have
\begin{equation*}
 \dim_{H}\Big(F(\alpha)\cap G(\beta)\Big)=\left\{ \begin{array}{ll}
                            1, & \ \ \ \text{if }~ \beta=0;\\
                            s\Big(\frac{\beta^{2}(1-\alpha)}{\beta-\alpha},{\frac{\sqrt{5}+1}{2}}\Big), & \ \ \ \text{if }~ 0\leq \alpha\leq \frac{\beta}{1+\beta}<\beta\leq1; \\
                            0, & \ \ \ \text{otherwise.}\end{array}\right.
       \end{equation*}
\end{thm}

\begin{thm}\label{thm2}
For $\alpha\in[0,1],$ we have
\begin{equation*}
 \dim_{H}F(\alpha)=\left\{\begin{array}{ll}
                                            s\big(4\alpha(1-\alpha),{\frac{\sqrt{5}+1}{2}}\big), & \ \ \ \text{if }~ 0\leq \alpha\leq \frac{1}{2}; \\
                                            0, & \ \ \ \text{otherwise.} \end{array}\right.
         \end{equation*}
         \end{thm}

\section{Preliminaries}

\subsection{Properties of continued fractions}

This section is devoted to recalling some elementary properties in continued fractions.
For more information on the continued fraction expansion, the readers are referred to \cite{HW,K,S}.
We also introduce some basic techniques for estimating the Hausdorff dimension of a fractal set (see \cite{F,SF}).

For any irrational number $x\in[0,1)$ with continued fraction expansion (\ref{e2}), we write
            $\frac{p_{n}(x)}{q_{n}(x)}=[a_{1}(x),\ldots,a_{n}(x)]$
and call it the $n$-th convergent of $x.$ With the conventions
           $p_{-1}(x)=1,$ $q_{-1}(x)=0,$ $p_{0}(x)=0$ and $q_{0}(x)=1,$
we know that $p_{n}(x)$ and $q_{n}(x)$ satisfy the recursive relations \cite{K}$\colon$
\begin{equation}\label{equ21}
  p_{n+1}(x)=a_{n+1}(x)p_{n}(x)+p_{n-1}(x),~~q_{n+1}(x)=a_{n+1}(x)q_{n}(x)+q_{n-1}(x), ~~n\geq 0.
\end{equation}
Clearly, $q_{n}(x)$ is determined by $a_{1}(x),\ldots,a_{n}(x),$ so we also write
            $q_{n}(a_{1}(x),\ldots,a_{n}(x))$
instead of $q_n(x)$. We write $a_{n}$ and $q_n$ in place of $a_{n}(x)$ and $q_n(x)$ for simplicity when no confusion can arise.

\begin{lem} [\cite{K}]\label{lem21}
For  $n\geq 1$ and $(a_{1},\ldots,a_{n})\in \mathbb{N}^{n}$, we have$\colon$

{\rm{(1)}} $q_{n}\geq 2^{\frac{n-1}{2}},$ and $\prod\limits_{k=1}^{n}a_{k}\leq q_{n}\leq \prod\limits_{k=1}^{n}(a_{k}+1).$

{\rm{(2)}} For any  $k\ge 1,$
\begin{equation*}
1\leq \frac{q_{n+k}(a_{1},\ldots,a_{n},a_{n+1},\ldots,a_{n+k})}
{q_{n}(a_{1},\ldots,a_{n})q_{k}(a_{n+1},\ldots,a_{n+k})}\leq 2.
\end{equation*}

{\rm{(3)}}
If $a_1=a_2=\cdots=a_{n}=i$, then
$$\frac{\big(\tau(i)\big)^{n}}{2} \leq q_{n}(i,\ldots,i)
=\frac{\big(\tau(i)\big)^{n+1}-\big(\zeta(i)\big)^{n+1}}{\tau(i)-\zeta(i)}\leq 2\big(\tau(i)\big)^{n},$$
where $\tau(i)=\frac{i+\sqrt{i^{2}+4}}{2}$ and $\zeta(i)=\frac{i-\sqrt{i^{2}+4}}{2}.$
\end{lem}

\begin{proof}
For the convenience of readers, we give the proof.

{\rm{(1)}}
By the recursive relations (\ref{equ21}), we readily check that
         $$\prod\limits_{k=1}^{n}a_{k}\leq q_{n}\leq \prod\limits_{k=1}^{n}(a_{k}+1).$$
Since $a_n\ge 1$ for $n\ge 1,$ we have
         $$1=q_0\le q_1<q_2<\cdots q_{n-1}<q_n.$$
By induction $q_n\ge2^{\frac{n-1}{2}}$  for all $n\ge1$;  similarly $p_n\ge2^{\frac{n-1}{2}}$.
\smallskip

{\rm{(2)}}
Induction on $k$:
assuming that
\begin{equation*}
1\leq
\frac{q_{n+k}(a_{1},\ldots,a_{n},a_{n+1},\ldots,a_{n+k})}{q_{n}(a_{1},\ldots,a_{n})q_{k}(a_{n+1},\ldots,a_{n+k})}
\leq 2
\end{equation*}
holds for all $k\in\{1,\ldots,m\}$, we prove that the above inequality holds for $k=m+1.$
Indeed, this is the case because
\begin{align*}
 &\ \ \ \ q_{n+m+1}(a_{1},\ldots,a_{n},a_{n+1},\ldots,a_{n+m+1})\\
 &=a_{n+m+1}q_{n+m}(a_{1},\ldots,a_{n},a_{n+1},\ldots,a_{n+m})+q_{n+m-1}(a_{1},\ldots,a_{n},a_{n+1},\ldots,a_{n+m-1})\\
 &\ge a_{n+m+1}q_{n}(a_{1},\ldots,a_{n})q_{m}(a_{n+1},\ldots,a_{n+m})+q_{n}(a_{1},\ldots,a_{n})q_{m-1}(a_{n+1},\ldots,a_{n+m-1})\\
 &=q_{n}(a_{1},\ldots,a_{n})q_{m+1}(a_{n+1},\ldots,a_{n+m+1}),
 \end{align*}
 and
 \begin{align*}
 &\ \ \ \ q_{n+m+1}(a_{1},\ldots,a_{n},a_{n+1},\ldots,a_{n+m+1})\\
 &=a_{n+m+1}q_{n+m}(a_{1},\ldots,a_{n},a_{n+1},\ldots,a_{n+m})+q_{n+m-1}(a_{1},\ldots,a_{n},a_{n+1},\ldots,a_{n+m-1})\\
 &\le 2a_{n+m+1}q_{n}(a_{1},\ldots,a_{n})q_{m}(a_{n+1},\ldots,a_{n+m})+2q_{n}(a_{1},\ldots,a_{n})q_{m-1}(a_{n+1},\ldots,a_{n+m-1})\\
 &=2q_{n}(a_{1},\ldots,a_{n})q_{m+1}(a_{n+1},\ldots,a_{n+m+1}).
 \end{align*}

{\rm{(3)}}
By the recursive relations (\ref{equ21}), we   deduce that
\begin{align*}
 \left(
  \begin{array}{cc}
    p_{n+1} & p_{n} \\
    q_{n+1} & q_{n} \\
  \end{array}
\right)&=\left(
          \begin{array}{cc}
            p_{n} & p_{n-1} \\
            q_{n} & q_{n-1} \\
          \end{array}
        \right)\left(
                 \begin{array}{cc}
                   a_{n+1} & 1 \\
                   1 & 0 \\
                 \end{array}
               \right)\\ &
=\left(
\begin{array}{cc}
 p_0 & p_{-1} \\
  q_0 & q_{-1} \\
   \end{array}
\right)\left(
         \begin{array}{cc}
           a_{1} & 1 \\
           1 & 0 \\
         \end{array}
       \right)\cdots\left(
                 \begin{array}{cc}
                   a_{n+1} & 1 \\
                   1 & 0 \\
                 \end{array}
               \right)\\&=
               \left(
\begin{array}{cc}
 0 & 1 \\
  1 & 0 \\
   \end{array}
\right)\left(
         \begin{array}{cc}
           a_{1} & 1 \\
           1 & 0 \\
         \end{array}
       \right)\cdots\left(
                 \begin{array}{cc}
                   a_{n+1} & 1 \\
                   1 & 0 \\
                 \end{array}
               \right).
\end{align*}
Taking $a_1=\cdots=a_n=a_{n+1}=i$ yields that
$$\left(
  \begin{array}{cc}
    p_{n+1} & p_{n} \\
    q_{n+1} & q_{n} \\
  \end{array}
\right)=\left(
\begin{array}{cc}
 0 & 1 \\
  1 & 0 \\
   \end{array}
\right)\left(
         \begin{array}{cc}
           i & 1 \\
           1 & 0 \\
         \end{array}
       \right)\cdots\left(
                 \begin{array}{cc}
                   i & 1 \\
                   1 & 0 \\
                 \end{array}
               \right).$$
     The    symmetric matrix $A=\left(
\begin{array}{cc}
 i & 1 \\
  1 & 0 \\
   \end{array}
\right)$ is diagonalizable: $$P^{-1}AP=\left(
  \begin{array}{cc}
   \tau(i) & 0 \\
    0 & \zeta(i) \\
  \end{array}
\right)$$ with $P=\left(
  \begin{array}{cc}
   \tau(i) & \zeta(i) \\
    1 & 1 \\
  \end{array}
\right)$.

A direct calculation yields that
      $$q_n(i,\ldots,i)
       =\frac{\big(\tau(i)\big)^{n+1}-\big(\zeta(i)\big)^{n+1}}{\tau(i)-\zeta(i)}.$$

Also
    $$\frac{\big(\tau(i)\big)^{n+1}-\big(\zeta(i)\big)^{n+1}}{\tau(i)-\zeta(i)}\le\frac{2\big(\tau(i)\big)^{n+1}}{\tau(i)}=2\big(\tau(i)\big)^{n},$$
   and, if $n$ is even,
       $$\frac{\big(\tau(i)\big)^{n+1}-\big(\zeta(i)\big)^{n+1}}{\tau(i)-\zeta(i)}\ge\frac{\big(\tau(i)\big)^{n+1}}{2\tau(i)}=\frac{\big(\tau(i)\big)^{n}}{2};$$
if $n$ is odd,   (since  $\zeta(i)\cdot\tau(i)=-1$)
\begin{align*}
  \frac{\big(\tau(i)\big)^{n+1}-\big(\zeta(i)\big)^{n+1}}{\tau(i)-\zeta(i)}
  =\frac{\big(\tau(i)\big)^{2(n+1)}-1}{\big(\tau(i)\big)^{n+2}+\big(\tau(i)\big)^{n}}
  \ge\frac{\big(\tau(i)\big)^{n}}{2}.
\end{align*}
This completes the proof.
\end{proof}

For  $n\geq 1$ and $(a_{1},\ldots,a_{n})\in \mathbb{N}^{n}$, we write
       $$I_{n}(a_{1},\ldots,a_{n})=\{x\in[0,1)\colon a_{k}(x)=a_{k}, 1\leq k\leq n\},$$
and call it a basic interval of order $n$.
The basic interval of order $n$ which contains $x$ will be denoted by  $I_{n}(x)$, i.e.,
       $I_{n}(x)=I_{n}(a_{1}(x),\ldots,a_{n}(x))$.

\begin{lem} [\cite{K}]\label{lem22}
For $n\geq 1$ and $(a_{1},\ldots,a_{n})\in \mathbb{N}^{n},$ we have
\begin{equation}\label{e22}
\frac{1}{2q_{n}^{2}}\leq|I_{n}(a_{1},\ldots,a_{n})|=\frac{1}{q_{n}(q_{n}+q_{n+1})}\leq\frac{1}{q_{n}^{2}}.
\end{equation}
Here and hereafter $|\cdot|$ denotes the length of an interval.
\end{lem}

The next lemma describes the distribution of basic intervals $I_{n+1}$ of order $n+1$ inside an $n$-th  basic interval $I_{n}.$

\begin{lem}[\cite{K}]\label{lem23}
Let $I_{n}(a_{1},\ldots,a_{n})$ be a basic interval of order $n,$ which is partitioned into sub-intervals $I_{n+1}(a_{1},\ldots,a_{n},a_{n+1})$ with $a_{n+1}\in \mathbb{N}.$
When $n$ is odd, these sub-intervals are positioned from left to right, as $a_{n+1}$ increases;
when $n$ is even, they are positioned from right to left.
\end{lem}

The following lemma displays  the relationship between the ball $B(x,|I_{n}(x)|)$ and the basic interval $I_{n}(x)$.

\begin{lem} [\cite{BW}]\label{lem24}
Let $x=[a_{1},a_{2},\ldots].$ We have$\colon$

{\rm{(1)}} if $a_{n}\neq 1,$ then $B(x,|I_{n}(x)|)\subset\bigcup\limits_{j=-1}^{3}I_{n}(a_{1},\ldots,a_{n}+j);$

{\rm{(2)}} if $a_{n}=1$ and $a_{n-1}\neq 1,$ then $B(x,|I_{n}(x)|)\subset\bigcup\limits_{j=-1}^{3}I_{n-1}(a_{1},\ldots,a_{n-1}+j);$

{\rm{(3)}} if $a_{n}=1$ and $a_{n-1}=1,$ then $B(x,|I_{n}(x)|)\subset I_{n-2}(a_{1},\ldots,a_{n-2}).$
\end{lem}

The following two properties, namely, H\"{o}lder property and the mass distribution principle, are often used to estimate the Hausdorff dimension of a fractal set.

\begin{lem} [\cite{F}] \label{lem212}
If $f\colon X \to Y$ is an $\alpha$-H\"{o}lder mapping between metric spaces,
that is, there exists $c>0$ such that for all $x_{1},x_{2}\in X$,
$$d(f(x_{1}),f(x_{2}))\leq cd(x_{1},x_{2})^{\alpha}.$$
Then $\dim_{H}f(X)\leq \frac{1}{\alpha}\dim_{H}X.$
\end{lem}

\begin{lem} [\cite{F}] \label{lem213}
Let $E\subseteq [0,1]$ be a Borel set and $\mu$
be a measure with $\mu(E)> 0.$ If for every $x \in E$,
$$\liminf_{r \to 0}\frac{\log \mu(B(x,r))}{\log r} \geq s,$$
then $ \dim_{H}E\geq s.$
\end{lem}

We conclude this subsection by quoting a dimensional result related to continued fractions, which will be used in the proof of Theorem \ref{TZ2}.

Let $\mathbf{K}=\{k_{n}\}_{n= 1}^{\infty}$ be a subsequence of $\mathbb{N}$ which is not cofinite.
Let $x=[a_{1},a_{2},\ldots]$ be an irrational number in $[0,1)$.
Eliminating all the terms  $a_{k_n}$ from the sequence $a_1,a_2,\ldots$,
we obtain an infinite subsequence $c_1,c_2,\ldots$, and put $\phi_{\mathbf{K}}(x)=y$ with $y=[c_1,c_2,\ldots]$.
In this way, we define a mapping $\phi_{\mathbf{K}}\colon [0,1)\cap \mathbb{Q}^{c}\to [0,1)\cap \mathbb{Q}^{c}$.

Let $\{M_{n}\}_{n\geq 1}$ be a sequence with $M_{n}\in\mathbb{N}$, $n\geq 1$. Set
$$S(\{M_{n}\})=\big\{x\in [0,1)\cap \mathbb{Q}^{c}\colon 1\leq a_{n}(x)\leq M_{n}~\text{for all}~n\geq 1\big\}.$$

\begin{lem}[\cite{CC}]\label{lem214}
Suppose that $\{M_{n}\}_{n=1}^{\infty}$ is a bounded sequence.
If the sequence $\mathbf{K}=\{k_{n}\}_{n=1}^{\infty}$ is of density zero in $\mathbb{N}$, then
$$\dim_{H}S(\{M_{n}\})=\dim_{H}\phi_{\mathbf{K}}S(\{M_{n}\}).$$
\end{lem}

\subsection{Pressure function and pre-dimensional number}
We now introduce the notions of the pressure function and pre-dimensional number in the continued fraction dynamical system.
For more details, we refer the reader to \cite{HMU}.

For $\mathcal{A}$ a finite or infinite subset of $\mathbb{N},$ we set
$$X_{\mathcal{A}}=\big\{x\in[0,1)\colon a_{n}(x)\in \mathcal{A} \text{ for all } n\geq 1\big\}.$$
The pressure function restricted to the subsystem $(X_{\mathcal{A}},T)$ with potential $\phi\colon [0,1)\to \mathbb{R} $ is defined as
\begin{equation}\label{e23}
P_{\mathcal{A}}(T,\phi)=\lim_{n\to\infty}\frac{\log\sum\limits_{(a_{1},\ldots,a_{n})\in \mathcal{A}^{n}}\sup\limits_{x\in X_{\mathcal{A}}}\exp{S_n \phi([a_{1},\ldots,a_{n}+x])}}{n},
\end{equation}
where $S_{n}\phi(x)=\phi(x)+\cdots+\phi(T^{n-1}(x))$ denotes the ergodic sum of $\phi$.
When $\mathcal{A}=\mathbb{N}$, we write $P(T,\phi)$ for $P_{\mathbb{N}}(T,\phi)$.

The $n$-th variation $\textrm{Var}_{n}(\phi)$ of $\phi$ is defined as
$$\textrm{Var}_{n}(\phi)=\sup\big\{|\phi(x)-\phi(y)|\colon I_{n}(x)=I_{n}(y)\big\}.$$
The following lemma shows the existence of the limit in (\ref{e23}).

\begin{lem} [\cite{PW}]
The limit defining $P_{\mathcal{A}}(T,\phi)$  in (\ref{e23}) exists. Moreover,
if $\phi\colon [0,1)\rightarrow \mathbb{R}$ satisfies $\textrm{Var}_{1}(\phi)<\infty$ and $\textrm{Var}_{n}(\phi)\rightarrow 0$ as $n\rightarrow \infty$,
  the value of $P_{\mathcal{A}}(T,\phi)$ remains the same even without taking the supremum over $x\in X_{\mathcal{A}}$ in (\ref{e23}).
\end{lem}

For $0<\alpha<1$ and $i\in \mathbb{N}$, we define 
\begin{equation*}
\widehat{s}_{n}(\mathcal{A},\alpha,{\tau(i)})
= \inf \left\{\rho \geq 0\colon \sum\limits_{a_{1},\ldots,a_{n}\in \mathcal{A}} \Big(\frac{1}{(\tau(i))^{\frac{n\alpha}{1-\alpha}}q_{n}(a_{1},\ldots,a_{n})}\Big)^{2\rho}\leq 1\right\}.
\end{equation*}
Following \cite{WW}, we call $\widehat{s}_{n}(\mathcal{A},\alpha,{\tau(i)})$  the $n$-th pre-dimensional number with respect to $\mathcal{A}$ and $\alpha$.  The properties of pre-dimensional numbers are  presented in the following lemmas; the original ideas for the proofs date back to Good \cite{G} (see also \cite{M1}).

\begin{lem}[\cite{WW}]\label{lem26}
Let $\mathcal{A}$ be a finite or infinite subset of $\mathbb{N}.$ For $0<\alpha<1$ and $i\in \mathbb{N}$,
the limit $\lim_{n\to\infty}\widehat{s}_{n}(\mathcal{A},\alpha,{\tau(i)})$ exists, denoted by $s(\mathcal{A},\alpha,{\tau(i)})$.
\end{lem}

By (\ref{e22}) and the definition of $\widehat{s}_{n}(\mathcal{A},\alpha,{\tau(i)}),$ we know
$0\le\widehat{s}_{n}(\mathcal{A},\alpha,{\tau(i)})\le1.$ Furthermore, Lemma \ref{lem26} implies that $0\le s(\mathcal{A},\alpha,{\tau(i)})\le1.$

\begin{lem}[\cite{WW}]\label{lem27}
 For any $B\in \mathbb{N},$ put $\mathcal{A}_{B}=\{1,\ldots,B\}.$
The limit  $\lim_{B \to \infty}s(\mathcal{A}_{B},\alpha,{\tau(i)})$ exists, and is  equal to $s(\mathbb{N},\alpha,{\tau(i)})$.
\end{lem}

Similarly to pre-dimensional numbers $\{\widehat{s}_{n}(\mathcal{A},\alpha,{\tau(i)}\},$  we define
\begin{equation*}
s_{n}(\mathcal{A},\alpha,{\tau(i)})=
\inf \left\{\rho \geq 0\colon \sum\limits_{a_{1},\ldots,a_{n-\lfloor n\alpha\rfloor}\in \mathcal{A}}
               \Big(\frac{1}{q_{n}(a_{1},\ldots,a_{n-\lfloor na\rfloor},i,\ldots,i)}\Big)^{2\rho}\leq 1\right\}.
\end{equation*}
\begin{rem}\label{rem2}
We remark that
$$\sum\limits_{a_{1},\ldots,a_{n}\in \mathcal{A}} \Big(\frac{1}{(\tau(i))^{\frac{n\alpha}{1-\alpha}}q_{n}(a_{1},\ldots,a_{n})}\Big)^{2\widehat{s}_{n}(\mathcal{A},\alpha,{\tau(i)})}\leq 1$$
and
$$\sum_{a_{1},\ldots,a_{n-\lfloor n\alpha\rfloor}\in \mathcal{A}}
\Big(\frac{1}{q_{n}(a_{1},\ldots,a_{n-\lfloor n\alpha\rfloor},i,\ldots,i)}\Big)^{2s_{n}(\mathcal{A},\alpha,{\tau(i)})}\leq1,$$
with equalities holding when $\mathcal{A}$ is finite.
\end{rem}
By Lemmas \ref{lem26} and \ref{lem27}, we have the following result.
\begin{lem} \label{lem29}
Let $\mathcal{A}$ be a finite or infinite subset of $\mathbb{N}.$ For $0<\alpha<1$ and $i\in \mathbb{N},$  we have
$$\lim_{n\to\infty}s_{n}(\mathcal{A},\alpha,{\tau(i)})=s(\mathcal{A},\alpha,{\tau(i)}).$$
In particular, if $\mathcal{A}=\mathbb{N},$ then
$$\lim_{n\to\infty}s_{n}(\mathbb{N},\alpha,{\tau(i)})=s(\mathbb{N},\alpha,{\tau(i)}).$$
\end{lem}
\begin{proof}
For $\varepsilon>0$ and $n$ large enough, we have
\begin{equation}\label{equ2.4}
2^{\frac{n-\lfloor n\alpha\rfloor}{2}\varepsilon}>64,
\end{equation}
\begin{equation}\label{equ2.5}
\frac{3}{(1-\alpha)(n\alpha-1)}+\frac{\log 4}{n\alpha-1}<\varepsilon,
\end{equation}
\begin{equation}\label{equ2.6}
|\widehat{s}_{n}(\mathcal{A},\alpha,{\tau(i)})-s(\mathcal{A},\alpha,{\tau(i)})|<\frac{\varepsilon}{2}.
\end{equation}
On the one hand, by Remark \ref{rem2},   we deduce that
\begin{align*}
  1 \ge& \sum_{a_1,\ldots,a_{n-\lfloor n\alpha\rfloor}\in \mathcal{A}}\Big(\frac{1}{q_{n}(a_1,\ldots,a_{n-\lfloor n\alpha\rfloor},i,\ldots,i)}\Big)^{2s_{n}(\mathcal{A},\alpha,{\tau(i)})} \\
    \ge & \sum_{a_1,\ldots,a_{n-\lfloor n\alpha\rfloor}\in \mathcal{A}}
            \Big(\frac{1}{2q_{n-\lfloor n\alpha\rfloor}(a_1,\ldots,a_{n-\lfloor n\alpha\rfloor})q_{\lfloor n\alpha\rfloor}(i,\ldots,i)}\Big)^{2s_{n}(\mathcal{A},\alpha,{\tau(i)})}\\
    \ge & \sum_{a_1,\ldots,a_{n-\lfloor n\alpha\rfloor}\in \mathcal{A}}\Big(\frac{1}{4q_{n-\lfloor n\alpha\rfloor}(a_1,\ldots,a_{n-\lfloor n\alpha\rfloor})
                (\tau(i))^{\frac{\alpha}{1-\alpha}(n-\lfloor n\alpha\rfloor)}}\Big)^{2s_{n}(\mathcal{A},\alpha,{\tau(i)})}\\
     \ge & \sum_{a_1,\ldots,a_{n-\lfloor n\alpha\rfloor}\in \mathcal{A}}\Big(\frac{1}{q_{n-\lfloor n\alpha\rfloor}(a_1,\ldots,a_{n-\lfloor n\alpha\rfloor})
                  (\tau(i))^{\frac{\alpha}{1-\alpha}(n-\lfloor n\alpha\rfloor)}}\Big)^{2s_{n}(\mathcal{A},\alpha,{\tau(i)})+\varepsilon},
\end{align*}
where the second inequality holds by Lemma \ref{lem21}(2);
the third inequality is right by Lemma \ref{lem21}(3) and the fact that $\frac{\alpha}{1-\alpha}(n-\lfloor n\alpha\rfloor)\ge \lfloor n\alpha\rfloor$ for $n\in \mathbb{N};$
the last inequality is true by Lemma \ref{lem21}(1) and (\ref{equ2.4}). This means that
$$s_{n}(\mathcal{A},\alpha,{\tau(i)})+\frac{\varepsilon}{2}\ge \widehat{s}_{n-\lfloor n\alpha\rfloor}(\mathcal{A},\alpha,{\tau(i)}).$$

On the other hand,   we have
\begin{align*}
 1 \ge & \sum_{a_1,\ldots,a_{n-\lfloor n\alpha\rfloor}\in \mathcal{A}}\Big(\frac{1}{q_{n-\lfloor n\alpha\rfloor}(a_1,\ldots,a_{n-\lfloor n\alpha\rfloor})
                    (\tau(i))^{\frac{\alpha}{1-\alpha}(n-\lfloor n\alpha\rfloor)}}\Big)^{2\widehat{s}_{n-\lfloor n\alpha\rfloor}(\mathcal{A},\alpha,{\tau(i)})}\\
  \ge & \sum_{a_1,\ldots,a_{n-\lfloor n\alpha\rfloor}\in \mathcal{A}}\Big(\frac{1}{q_{n-\lfloor n\alpha\rfloor}(a_1,\ldots,a_{n-\lfloor n\alpha\rfloor})
                    (\tau(i))^{\lfloor n\alpha\rfloor+\frac{1}{1-\alpha}}}\Big)^{2\widehat{s}_{n-\lfloor n\alpha\rfloor}(\mathcal{A},\alpha,{\tau(i)})}\\
  \ge & \sum_{a_1,\ldots,a_{n-\lfloor n\alpha\rfloor}\in \mathcal{A}}\Big(\frac{1}{q_{n}(a_1,\ldots,a_{n-\lfloor n\alpha\rfloor},i,\ldots,i)}\Big)^{2\widehat{s}_{n-\lfloor n\alpha\rfloor}(\mathcal{A},\alpha,{\tau(i)})+\varepsilon},
\end{align*}
where the second inequality is obtained by $\frac{\alpha}{1-\alpha}(n-\lfloor n\alpha\rfloor)\le \lfloor n\alpha\rfloor+\frac{1}{1-\alpha}$ for $n\in \mathbb{N}$;
      the last inequality holds by Lemma \ref{lem21}(3) and (\ref{equ2.5}). This implies that
$$s_{n}(\mathcal{A},\alpha,{\tau(i)})\le \widehat{s}_{n-\lfloor n\alpha\rfloor}(\mathcal{A},\alpha,{\tau(i)})+\frac{\varepsilon}{2}.$$

Thus, by (\ref{equ2.6}), we obtain that
$$|s_{n}(\mathcal{A},\alpha,{\tau(i)})-s(\mathcal{A},\alpha,{\tau(i)})|<\varepsilon$$
for $n$ large enough. This completes the proof.
\end{proof}

For simplicity, write $s_{n}(\alpha,{\tau(i)})$ for $s_{n}(\mathbb{N},\alpha,{\tau(i)}),$ $s(\alpha,{\tau(i)})$ for $s(\mathbb{N},\alpha,{\tau(i)}).$
\begin{lem} [\cite{WW}] \label{lem210}
For $ 0<\alpha<1$ and $i\in \mathbb{N}$, we have$\colon$

{\rm{(1)}} $s(\alpha,{\tau(i)})>\frac{1}{2}$;

{\rm{(2)}} $s(\alpha,{\tau(i)})$ is non-increasing and continuous with respect to $\alpha$;

{\rm{(3)}} $\lim_{\alpha \to 0}{s(\alpha,{\tau(i)})}=1$ and $\lim_{\alpha \to 1}{s(\alpha,{\tau(i)})}=\frac{1}{2}.$
\end{lem}

From a point of view of dynamical system, $s(\alpha,{\tau(i)})$ can be regarded as the solution to the pressure function \cite{WWX}
$$P\Big(T,-s\Big(\log |T'|+\frac{\alpha}{1-\alpha}\log{\tau(i)}\Big)\Big)=0.$$
Furthermore, by Lemma \ref{lem210}, we may extend  $s(\alpha,{\tau(i)})$   to $[0,1]$ as follows$\colon$
\begin{equation}\label{equ2.7}
s(\alpha,{\tau(i)})=\left\{
      \begin{array}{ll}
        1, &  \ \ \ \alpha=0, \\
        s(\alpha,{\tau(i)}), & \ \ 0<\alpha <1, \\
        \frac{1}{2}, & \ \ \ \alpha=1.
      \end{array}
    \right.
\end{equation}

\section{Proof of Theorem \ref{TZ2}$\colon$ Upper~bound}\label{S_3}
In this section, we devote to estimating the upper bound of $E(\hat{\nu},\nu)$.

We first consider the case $\nu=0$.

\begin{lem}\label{full}
$\nu(x)=0$ for Lebesgue almost all $x\in[0,1).$
\end{lem}
\begin{proof}
Since $\sum_{n=1}^{\infty}|I_n(y)|^{\frac{1}{m}}<\infty$, we obtain by   Theorem 2B in \cite{P67}   that the set
$$\Big\{x\in[0,1)\colon |T^{n}(x)-y|<|I_{n}(y)|^{\frac{1}{m}} \text{ for infinitely many }n\in \mathbb{N}\Big\}$$ is of   measure zero. Now
\begin{align*}
\{x\in[0,1)\colon \nu(x)>0\}&\subseteq\bigcup_{m=1}^{\infty}\Big\{x\in[0,1)\colon \nu(x)>\frac{1}{m}\Big\}
                              \\&\subseteq\bigcup_{m=1}^{\infty}\Big\{x\in[0,1)\colon |T^{n}(x)-y|<|I_{n}(y)|^{\frac{1}{m}} \text{ for infinitely many }n\in \mathbb{N}\Big\}.
\end{align*}
Hence $\{x\in[0,1)\colon \nu(x)>0\}$ is a null set.
This completes the proof.
\end{proof}

We now aim to determine the upper bound of $\dim_{H}E(\hat{\nu},\nu)$ for $0<\nu\le+\infty.$

\begin{lem}\label{lem31}
Let $x\in E(\hat{\nu},\nu)$, where $v>0$. If the continued fraction expansion of $x$ is not periodic,  there exist two ascending sequences $\{n_k\}_{k=1}^{\infty}$ and $\{m_k\}_{k=1}^{\infty}$ depending on $x$ such that$\colon$

{\rm{(1)}} $n_k<m_k<n_{k+1}<m_{k+1}$ for $k\ge1$;

{\rm{(2)}} $a_{n_k+1}(x)=\cdots=a_{m_k}(x)=i$ for $k\ge1;$

{\rm{(3)}} $\displaystyle\liminf_{k\to\infty}\frac{m_k-n_k}{n_{k+1}}=\hat{\nu},$ $\displaystyle\limsup_{k\to\infty}\frac{m_k-n_k}{n_k}=\nu.$
\end{lem}

\begin{proof}
For $x=[a_{1}(x),a_{2}(x),\ldots]\in E(\hat{\nu},\nu)$, we define two sequences $\{n_{k}'\}_{k\geq1}$ and $\{m_{k}'\}_{k\geq 1}$  as follows$\colon$
\begin{align*}
  m_{0}'=0,~ n_k'&=\min\{n\ge m_{k-1}'\colon a_{n+1}(x)=i\},\\
   m_k'&=\max\{n\ge n_k'\colon a_{n_k'+1}=\cdots=a_{n}(x)=i\}.
\end{align*}
The fact that $\nu(x)>0$ guarantees the existence of $n_k'$, and thus $m_k'$ is well defined since
the  continued fraction expansion of $x$ is not periodic.
Further, for all $k\ge 1,$ we have that $n_k'\le m_k'<n_{k+1}'$, and
  $$\frac{1}{2(i+2)^{2}}|I_{m_k'-n_k'}(y)|\le|T^{n_k'}(x)-y|<|I_{m_k'-n_k'}(y)|,$$
where the first inequality holds by Lemma \ref{lem23}.

We also have $\limsup_{k\to\infty}(m_{k}'-n_{k}')= +\infty$ since $\nu(x)>0$. We then choose a subsequence of $\{(n_k', m_k')\}_{k\geq 1}$ as follows:  put $(n_{1},m_{1})=(n'_{1},m'_{1});$
having choosen $(n_{k}, m_k)=(n_{j_{k}}', m_{j_{k}}')$, we set
$j_{k+1}=\min\big\{j>j_k\colon m_{j}'-n_{j}'> m_{k}-n_{k}\big\},$
and put $(n_{k+1},m_{k+1})=(n_{j_{k+1}}', m_{j_{k+1}}')$. 
 We claim that
$$ \liminf_{k\to\infty}\frac{m_k-n_k}{n_{k+1}}=\hat{\nu}(x),\ \ \limsup_{k\to\infty}\frac{m_k-n_k}{n_k}=\nu(x).$$
To prove the first assertion, we write
$\liminf_{n\to\infty}\frac{m_k-n_k}{n_{k+1}}=a.$
For  $\varepsilon>0,$ there is a subsequence $\{k_j\}_{j=1}^{\infty}$ such that
$$m_{k_j}-n_{k_{j}}\le (a+\varepsilon)n_{{k_j}+1}.$$
Putting $N=n_{k_j}-1$, we have for all $n\in [1,N]$ that
$$|T^{n}(x)-y|\ge\frac{1}{2(i+2)^{2}}|I_{m_{k_j}-n_{k_j}}(y)|>|I_{n_{k_j+1}}(y)|^{a+2\varepsilon}>|I_{N}(y)|^{a+3\varepsilon},$$
where the second inequality holds by the fact $\lim_{n\to\infty}\frac{-\log|I_n(y)|}{2n}=\log\tau(i).$
We deduce that $\hat{\nu}(x)\le a+3\varepsilon$ by the definition of $\hat{\nu}(x)$.

On the other hand,  when $k\gg1$, we have
 $$m_{k}-n_{k}\ge (a-\varepsilon)n_{k+1}.$$
   For  $n_k\le N<n_{k+1}$,
$$|T^{n_{k}}(x)-y|\le|I_{m_k-n_k}(y)|<|I_{n_{k+1}}(y)|^{a-\varepsilon}<|I_{N}(y)|^{a-\varepsilon}.$$
From here we deduce that   $\hat{\nu}(x)\ge a-\varepsilon$. 

Letting $\varepsilon\to 0$ we complete the proof of the first assertion; the second one can be proved in a similar way.
\end{proof}

\begin{lem}\label{lem32}
If $0<\frac{\nu}{1+\nu}<\hat{\nu}\le\infty,$ $E(\hat{\nu},\nu)$ is at most countable, and $\dim_{H}E(\hat{\nu},\nu)=0.$
\end{lem}

\begin{proof} If $x\in E(\hat{\nu},\nu)$, and its continued fraction expansion is not periodic, then
by Lemma \ref{lem31}(2),  there exist two sequences $\{n_k\}_{k=1}^{\infty}$ and $\{m_k\}_{k=1}^{\infty}$ depending on $x$ such that
$$\liminf_{k\to\infty}\frac{m_k-n_k}{n_{k+1}}=\hat{\nu},\ \ \limsup_{k\to\infty}\frac{m_k-n_k}{m_k}=\frac{\nu}{1+\nu}.$$
This yields $\hat{\nu}\le\frac{\nu}{1+\nu}$; the lemma follows.
\end{proof}

We devote to constructing a covering of $E(\hat{\nu},\nu)$ in  the case where $0\le\hat{\nu}\le\frac{\nu}{1+\nu}<\infty$ and $0<\nu\le\infty.$ Since  $E(0,\nu)$ is a subset of $\{x\in[0,1)\colon \nu(x)=\nu\},$ by Corollary \ref{Cor}, we have $\dim_{H}E(0,\nu)\le s\Big(\frac{\nu}{1+\nu},{\tau(i)}\Big)$, which is the desired upper bound estimate. Hence, we only need to deal with the case $0<\hat{\nu}\le\frac{\nu}{1+\nu}<\nu\le\infty.$
Whence, given any $x$  in the set $E(\hat{\nu},\nu)$ with non-periodic continued fraction expansion, we associate $x$ with two sequences $\{n_k\}, \{m_k\}$ as in Lemma \ref{lem31}.
The following properties hold:

(1) the sequence $\{m_k\}$ grows exponentially, more precisely,
there exists  $C>0$, independent of $x$, such that when $k$ is large enough,
\begin{equation}\label{e30}
k \leq C\log{m_{k}}.
\end{equation}

Indeed, we have that $m_{k}-n_{k}\ge(\hat{\nu}/2)n_{k+1}$
for all large $k$, and  thus
$$m_{k}\ge (1+\frac{\hat{\nu}}2)n_{k}
\ge (1+\frac{\hat{\nu}}2)m_{k-1}.$$

\medskip

(2) Write $\xi=\frac{\nu^{2}}{(1+\nu)(\nu-\hat{\nu})}$. For any  $\varepsilon>0$, there exist infinitely many $k$ such that
 \begin{equation}\label{e36}
 \sum_{i=1}^{k}(m_{i}-n_{i})\ge   m_{k}(\xi-\varepsilon). \end{equation}

To prove this, we apply a general form of the Stolz-Ces\`aro theorem which states that: if $b_n$
tends to infinity monotonically,
$$ \liminf_n\frac{a_n-a_{n-1}}{b_n-b_{n-1}}\le \liminf_n\frac{a_n}{b_n}\le \limsup_n\frac{a_n}{b_n} \le \limsup_n\frac{a_n-a_{n-1}}{b_n-b_{n-1}}.$$

  We deduce from   Lemma \ref{lem31}  that
\begin{equation*}
\limsup_{k\to\infty}\frac{m_{k}}{n_{k}}=1+\nu
\end{equation*}
and
\begin{equation*}
\liminf_{k\to\infty}\frac{m_{k}}{n_{k+1}}
\ge \liminf_{k\to\infty}\frac{m_{k}-n_{k}}{n_{k+1}}
\cdot\liminf_{k\to\infty}\frac{m_{k}}{m_{k}-n_{k}}
=\frac{\hat{\nu}(1+\nu)}{\nu}.
\end{equation*}
Hence
\begin{equation}\label{e34}
\begin{split}
  \liminf_k\frac{\sum_{i=1}^{k}(m_i-n_i)}{m_{k+1}}&\ge
  \liminf_k\frac{m_k-n_k}{m_{k+1}-m_k}\\&\ge
  \liminf_k\frac{m_k-n_k}{n_{k+1}}\cdot\frac{1}{\limsup_k\frac{m_{k+1}}{n_{k+1}}-\liminf_k\frac{m_k}{n_{k+1}}}\\&\ge
    \frac{\hat{\nu}\nu}{\big(\nu-\hat{\nu}\big)\big(1+\nu\big)},
\end{split}
\end{equation}
and thus
\begin{align*}
\sum_{i=1}^{k}(m_{i}-n_{i})&
\geq \Big(\frac{\hat{\nu}\nu}{\big(\nu-\hat{\nu}\big)\big(1+\nu\big)}-\frac{\varepsilon}{2}\Big)m_{k}+(m_{k}-n_{k})
\end{align*}
holds for $k$ large enough. On the other hand, there exist infinitely many $k$ such that
$$m_k-n_k\ge \Big(\frac{\nu}{1+\nu}-\frac{\varepsilon}{2}\Big)m_k.$$
 We then readily check that (\ref{e36}) holds for such $k$.

\medskip

We now construct a covering of $E(\hat{\nu},\nu)$. To this end,
we collect all   sequences  $(\{n_{k}\}, \{m_{k}\})$  associated with some $x\in E(\hat{\nu},\nu)$ as in Lemma \ref{lem31} to form a set
\begin{align*}
\Omega
= \Big\{(\{n_{k}\},\{m_{k}\})\  \colon \text{Conditions (1) \& (3) in Lemma \ref{lem31} are fulfilled} \Big\}.
\end{align*}
 For $(\{n_{k}\},\{m_{k}\})\in \Omega$, write
\begin{align*}
&H(\{n_{k}\},\{m_{k}\})=\{x\in[0,1)\colon  \text{Condition (2) in Lemma \ref{lem31} is fulfilled} \},\\
&\Lambda_{k, m_k}=\Big\{(n_1,m_1; \ldots; n_{k-1},m_{k-1};n_{k})\colon n_{1}<m_{1}<\cdots<m_{k-1}<n_{k}<m_k, (\ref{e36})\text{ holds} \Big\},\\
&\mathcal{D}_{n_1,m_1;\ldots;n_{k}, m_{k}}=\Big\{(\sigma_{1},\ldots,\sigma_{m_{k}})\in \mathbb{N}^{m_{k}}\colon \sigma_{n_{j}+1}=\cdots=\sigma_{m_{j}}=i \text{ for all }1\leq j\leq k \Big\}.
\end{align*}

Based on the previous analysis, we obtain a covering of $E(\hat{\nu},\nu),$ i.e.,
\begin{equation*}
\begin{split}
E(\hat{\nu},\nu) &\subseteq \bigcup_{(\{n_{k}\},\{m_{k}\})\in \Omega}H(\{n_{k}\},\{m_{k}\}) \\&
   \subseteq\bigcap_{K=1}^{\infty} \bigcup_{k=K}^{\infty}
            \bigcup_{m_{k}\ge e^{\frac{k}{C}}}
            \bigcup_{(n_{1},m_{1},\ldots,m_{k-1},n_{k})\in\Lambda_{k, m_k}}
            \bigcup_{(a_{1},\ldots,a_{m_{k}})\in \mathcal{D}_{n_1,m_1;\ldots;n_{k}, m_{k}}}I_{m_{k}}(a_{1},\ldots,a_{m_{k}}).
\end{split}
\end{equation*}
For $\varepsilon>0$, putting $t=s(\xi-\varepsilon,{\tau(i)})+\frac{\varepsilon}{2}$, we have that $t>s(\xi,{\tau(i)})$ and $t>\frac{1}{2}$ by Lemma \ref{lem210}(2) and (\ref{equ2.7}).
We are now in a position to  estimate $\mathcal{H}^{t+\frac{\varepsilon}{2}}\big(E(\hat{\nu},\nu)\big)$, the $(t+\frac{\varepsilon}{2})$-dimensional Hausdorff measure of $E(\hat{\nu},\nu)$.

\begin{lem}\label{lem35}
For $\varepsilon>0$, we have $\mathcal{H}^{t+\frac{\varepsilon}{2}}\big(E(\hat{\nu},\nu)\big)<+\infty.$
\end{lem}

\begin{proof}
By Lemma \ref{lem29}, there exists   $K\in \mathbb{N}$ such that for all $k\geq K,$
\begin{equation}\label{e35}
s_{m_{k}}(\xi-\varepsilon,{\tau(i)})\leq t,
\end{equation}
\begin{equation}\label{e37}
(4^{2t+\varepsilon}k)^{2C\log k}<2^{\frac{k-1}{4}\varepsilon}.
\end{equation}

 Writing $\psi(m_{k})=m_{k}-\sum_{i=1}^{k}(m_{i}-n_{i})$ when $m_{k}\geq K$, we have that
\begin{equation*}
\begin{split}
   &\sum\limits_{(a_{1},\ldots,a_{m_{k}})\in {\mathcal{D}_{n_1,m_1;\ldots;n_{k}, m_{k}}}}
                               |I_{m_{k}}(a_{1},\ldots,a_{m_{k}})|^{t+\frac{\varepsilon}{2}}\\
\le&\sum\limits _{(a_{1},\ldots,a_{m_{k}})\in {\mathcal{D}_{n_1,m_1;\ldots;n_{k}, m_{k}}}}
 \Big(\frac{1}{q_{m_{k}}(a_{1},\ldots,a_{m_{k}})}\Big)^{2t+\varepsilon}\\
\leq&\sum\limits _{a_{1},\ldots,a_{\psi(m_{k})}\in \mathbb{N}}2^{k(2t+\varepsilon)}\Big(\frac{1}{q_{\psi(m_{k})}(a_{1},\ldots,
                               a_{\psi(m_{k})})q_{m_{1}-n_{1}}(i,\ldots,i)\cdots q_{m_{k}-n_{k}}(i,\ldots,i)}\Big)^{2t+\varepsilon}\\
\leq&\sum\limits _{a_{1},\ldots,a_{\psi(m_{k})}\in \mathbb{N}}4^{k(2t+\varepsilon)}\Big(\frac{1}{q_{m_{k}}(a_{1},\ldots,a_{\psi(m_{k})},i,\ldots,i)}\Big)^{2t+\varepsilon}\\
\leq&\sum\limits _{a_{1},\ldots,a_{m_{k}-\lfloor m_{k}(\xi-\varepsilon)\rfloor}\in \mathbb{N}} 4^{k(2t+\varepsilon)}
                               \Big(\frac{1}{q_{m_{k}}(a_{1},\ldots,a_{m_{k}-\lfloor m_{k}(\xi-\varepsilon)\rfloor},i,\ldots,i)}\Big)^{2s_{m_{k}}(\xi-\varepsilon,{\tau(i)})+\varepsilon}\\
\leq&4^{k(2t+\varepsilon)}\Big(\frac{1}{2}\Big)^{\frac{m_{k}-1}{2}\varepsilon},
\end{split}
\end{equation*}
where the first two inequalities hold by Lemmas \ref{lem22} and \ref{lem21};
the penultimate one follows by (\ref{e36}) and (\ref{e35}); the last one follows by Remark \ref{rem2} and Lemma \ref{lem21}(1).

Therefore,
\begin{equation*}
\begin{split}
&\mathcal{H}^{t+\frac{\varepsilon}{2}}\big(E(\hat{\nu},\nu)\big)\\
   \leq  &\liminf_{K\to\infty}\sum\limits_{k=K}^{\infty}\sum_{m_{k}= e^{\frac{k}{C}}}^{\infty}
                      \sum_{(n_{1},m_{1},\ldots,m_{k-1},n_{k})\in\Lambda_{k, m_k}}
                      \sum\limits _{(a_{1},\ldots,a_{m_{k}})\in \mathcal{D}_{n_1,m_1;\ldots;n_{k}, m_{k}} }|I_{m_{k}}(a_{1},\ldots,a_{m_{k}})|^{t+\frac{\varepsilon}{2}}\\
\leq&\liminf_{K\to\infty}\sum\limits_{k=K}^{\infty}
                       \sum\limits_{m_{k}=e^{\frac{k}{C}}}^{\infty}
                       \sum_{n_{k}=1}^{m_{k}}\sum_{m_{k-1}=1}^{n_{k}}\cdots\sum\limits_{m_{1}=1}^{n_{2}}
                       \sum\limits_{n_{1}=1}^{m_{1}}4^{k(2t+\varepsilon)}\Big(\frac{1}{2}\Big)^{\frac{m_{k}-1}{2}\varepsilon}\\
\leq&\liminf_{K\to\infty}\sum\limits_{k=K}^{\infty}
                         \sum\limits_{m_{k}=e^{\frac{k}{C}}}^{\infty}(4^{2t+\varepsilon}m_{k})^{2C\log m_{k}}\Big(\frac{1}{2}\Big)^{\frac{m_{k}-1}{2}\varepsilon}\\
\leq&\liminf_{K\to\infty}\sum\limits_{k=K}^{\infty} \sum\limits_{m_{k}=e^{\frac{k}{C}}}^{\infty}\Big(\frac{1}{2}\Big)^{\frac{m_{k}-1}{4}\varepsilon}
                         \leq \frac{1}{1-(\frac{1}{2})^{\frac{\varepsilon}{4}}}
                         \sum\limits_{k=1}^{\infty} \Big(\frac{1}{2^{\varepsilon}}\Big)^{\frac{e^{\frac{k}{C}}-1}{4}} < +\infty,
\end{split}
\end{equation*}
where the third and fourth inequalities follow from (\ref{e30}) and  (\ref{e37}) respectively.
\end{proof}

By Lemma \ref{lem35}, we obtain the desired inequality $\dim_{H}E(\hat{\nu},\nu)\le s(\xi,{\tau(i)})$ by letting $\varepsilon\to 0$.

\section{Proof of Theorem \ref{TZ2}$\colon$ Lower~bound}\label{S_4}
In this section we establish  the lower bound of $\dim_{H}E(\hat{\nu},\nu).$ Since $E(0,0)$ is of full Lebesgue measure and $\dim_{H}E(\hat{\nu},\nu)=0$ for $\hat{\nu}>\frac{\nu}{1+\nu},$ we need only consider  the cases   $0\le\hat{\nu}\le\frac{\nu}{1+\nu}<\nu<\infty$ or $\nu=\infty.$

Let us start by treating the case  $0\le\hat{\nu}\le\frac{\nu}{1+\nu}<\nu<\infty$. We claim that there exist  two   sequences of natural numbers   $\{n_{k}\}_{k=1}^{\infty}$ and $\{m_{k}\}_{k=1}^{\infty}$
 satisfying the following conditions$\colon$

{\rm{(1)}} $n_{k}< m_{k}< n_{k+1}$ and $(m_{k}-n_{k})\leq(m_{k+1}-n_{k+1})$   for $k\geq 1$;

{\rm{(2)}} $\lim_{k\to\infty}\frac{m_{k}-n_{k}}{n_{k+1}}=\hat{\nu}$;

{\rm{(3)}} $\lim_{k\to\infty}\frac{m_{k}-n_{k}}{n_{k}}=\nu$.

\noindent Indeed, when $\hat{\nu}>0$, we may take
$$n_{1}=2,~ n_{k+1}=\left\lfloor \frac{\nu}{\hat{\nu}}
\Big(n_{k}+\frac{1}{\nu}\Big)\right\rfloor+2,~ m_{k}=\left\lfloor (1+\nu)n_{k}\right\rfloor+1;$$
when $\hat{\nu}=0$, we may take
$$n_{k}=\left\lfloor(1+\nu)2^{2^{2{k}}}\right\rfloor+2,~ m_{k}=\left\lfloor(1+\nu)n_{k}\right\rfloor+1.$$

From now on, we fix two such sequences  $\{n_{k}\}, \{m_{k}\}$; for any $B\ge i+1$, we define
\begin{equation*}
E(B)=\big\{x\in[0,1)\colon 1\leq a_{n}(x)\leq B, a_{n_{k}+1}(x)=\cdots=a_{m_{k}}(x)=i,
                                                           n\geq1 \text{ and } k\geq1\big\}.
\end{equation*}

The lower bound estimate of $\dim_H E(\hat\nu,\nu)$ will be established in the following way:
we  provide a lower bound of $\dim_{H}E(B)$;
build an injective mapping $f$ from $E(B)$ to $ E(\hat\nu,\nu)$ and prove that $f$ is dimension-preserving.

\subsection{Lower bound of $\dim_{H}E(B)$}\label{S_{4.1}}
Before proceeding, we cite an analogous definition of the pre-dimensional numbers.
Let $l_{k}=m_{k}-m_{k-1}$ for $k\geq 1$ ($m_{0}=0$ by convention). Let
$$\widetilde{f}_{k}(s,{\tau(i)})
                      =\sum\limits_{1\leq a_{m_{k-1}+1},\ldots,a_{n_{k}}\leq B}
                            \Big(\frac{1}{q_{l_{k}}(a_{m_{k-1}+1},\ldots,a_{n_{k}}, i,\ldots,i)}\Big)^{2s}.$$
We define $\widetilde{s}_{l_{k}}(\mathcal{A}_{B},\xi,\tau(i))$ to be the solution of the equation $\widetilde{f}_{k}(s,\tau(i))=1$.

\begin{lem}\label{lem41}
The limit $\lim_{k\to\infty}\widetilde{s}_{l_{k}}(\mathcal{A}_{B},\xi,{\tau(i)})$ exists, and is  equal to $s(\mathcal{A}_{B},\xi,{\tau(i)}).$
\end{lem}
\begin{proof}
 By   Lemma \ref{lem29} and the fact  $l_{k}\to\infty$ as $k\rightarrow\infty$ (cf. Condition (1)), we deduce that, for any $\varepsilon>0,$
when $k\gg 1,$
\begin{equation}\label{e43}
    |s_{l_{k}}(\mathcal{A}_{B},\xi+\varepsilon,\tau(i))-s(\mathcal{A}_{B},\xi+\varepsilon,\tau(i))|<\frac{\varepsilon}{2}.
\end{equation}
Further, from Conditions (2) and (3) we have that
\begin{align*}
  \lim_{k\to\infty}\frac{m_{k}-n_{k}}{l_{k}}
   =\lim_{k\to\infty}\frac{\frac{m_k-n_k}{m_k}\cdot\frac{m_k}{n_k}}{\frac{m_k}{n_k}-\frac{m_{k-1}-n_{k-1}}{n_k}\cdot\frac{m_{k-1}}{m_{k-1}-n_{k-1}}}
   =\xi,
\end{align*}
and thus for $k\gg 1$,
\begin{equation}\label{e42}
\lfloor l_{k}(\xi-\varepsilon)\rfloor
\leq m_{k}-n_{k}\leq \lfloor l_{k}(\xi+\varepsilon)\rfloor.
\end{equation}
Hence, by (\ref{e43}) and   (\ref{e42}), we obtain that
\begin{align*}
    &\sum\limits_{1\leq a_{1},\ldots,a_{n_{k}-m_{k-1}}\leq B} \Big(\frac{1}{q_{l_{k}}(a_{1},\ldots,a_{n_{k}-m_{k-1}},
                                                            i,\ldots,i)}\Big)^{2\big(s(\mathcal{A}_{B},
                                                            \xi+\varepsilon,\tau(i))-\varepsilon\big)} \\
\geq&\sum\limits_{1\leq a_{1},\ldots,a_{n_{k}-m_{k-1}}\leq B} \Big(\frac{1}{q_{l_{k}}(a_{1},\ldots,a_{n_{k}-m_{k-1}},
                                                                i,\ldots,i)}\Big)^{2s_{l_{k}}(\mathcal{A}_{B},
                                                                 \xi+\varepsilon,\tau(i))-\varepsilon}\\
\geq&\sum\limits_{1\leq a_{1},\ldots,a_{l_{k}-\lfloor l_{k}(\xi+\epsilon)\rfloor}\leq B}
                                                               \Big(\frac{1}{q_{l_{k}}(a_{1},\ldots,a_{l_{k}-\lfloor l_{k}(\xi+\varepsilon)\rfloor},
                                                                 i,\ldots,i)}\Big)^{2s_{l_{k}}(\mathcal{A}_{B},
                                                                  \xi+\varepsilon,{\tau(i)})-\varepsilon}\\
\geq&\Big(\frac{1}{q_{l_{k}}(B,\ldots,B)}\Big)^{-\varepsilon}\geq\tau(B)^{l_{k}\varepsilon}\geq 1
\end{align*}
and
\begin{align*}
    &\sum\limits_{1\leq a_{1},\ldots,a_{n_{k}-m_{k-1}}\leq B} \Big(\frac{1}{q_{l_{k}}(a_{1},\ldots,a_{n_{k}-m_{k-1}},
                                                              i,\ldots,i)}\Big)^{2\big(s(\mathcal{A}_{B},
                                                              \xi-\varepsilon,\tau(i))+\varepsilon\big)} \\
\leq&\sum\limits_{1\leq a_{1},\ldots,a_{l_{k}-\lfloor l_{k}(\xi-\varepsilon)\rfloor}\leq B}
                                                                \Big(\frac{1}{q_{l_{k}}(a_{1},\ldots,a_{l_{k}-\lfloor l_{k}(\xi-\varepsilon)\rfloor},
                                                                 i,\ldots,i)}\Big)^{2s_{l_{k}}(\mathcal{A}_{B},
                                                                  \xi-\varepsilon,\tau(i))+\varepsilon}\\
\leq&\Big(\frac{1}{q_{l_{k}}(1,\ldots,1)}\Big)^{\varepsilon}<1.
\end{align*}
By the monotonicity of $\widetilde{f}_{k}(s,{\tau(i)})$ with respect to $s$, we have
$$s(\mathcal{A}_{B},\xi+\varepsilon,\tau(i))-\varepsilon
\le\widetilde{s}_{l_{k}}(\mathcal{A}_{B},\xi,{\tau(i)})
\le s(\mathcal{A}_{B},\xi-\varepsilon,\tau(i))+\varepsilon,$$
which  completes the proof.
\end{proof}

\subsubsection{\textbf{Supporting measure.}}
We define a probability measure $\mu$  on $E(B)$ by distributing mass among the basic intervals.   We introduce the symbolic space to code these basic intervals: write $\mathcal{A}_{B}=\{1,\ldots,B\}$;  for $n\geq 1,$ set
\begin{equation*}
\mathcal{B}_{n}=\big\{(\sigma_{1},\ldots,\sigma_{n})\in \mathcal{A}_{B}^{n}\colon\sigma_{j}=i \text{ for } n_{k}<j\leq m_{k} \text{ with some } k\geq 1\big\}.
\end{equation*}


\noindent Step I$\colon$ For $(a_{1},\ldots,a_{m_{1}})\in \mathcal{B}_{m_{1}}$, we define
$$\mu\big(I_{m_{1}}(a_{1},\ldots,a_{m_{1}})\big)=\Big(\frac{1}{q_{l_{1}}(a_{1},\ldots,a_{m_{1}})}\Big)^{2\widetilde{s}_{l_{1}}(\mathcal{A}_{B},\xi,{\tau(i)})}$$
and for $1\leq n<m_{1}$, set
$$\mu\big(I_{n}(a_{1},\ldots,a_{n})\big)=\sum\limits_{a_{n+1},\ldots,a_{m_{1}}}\mu\big(I_{m_{1}}(a_{1},\ldots,a_{n},a_{n+1},\ldots,a_{m_{1}})\big),$$
where the summation is taken over all $(a_{n+1},\ldots,a_{m_{1}})$ with $(a_{1},\ldots,a_{m_{1}})\in \mathcal{B}_{m_{1}}$.

\noindent Step II$\colon$ Assuming that $\mu\big(I_{m_{k}}(a_{1},\ldots,a_{m_{k}})\big)$ is defined for some $k\geq 1$, we define
$$\mu\big(I_{m_{k+1}}(a_{1},\ldots,a_{m_{k+1}})\big)=\mu\big(I_{m_{k}}(a_{1},\ldots,a_{m_{k}})\big)\cdot \Big(\frac{1}{q_{l_{k+1}}(a_{m_{k}+1},\ldots,a_{m_{k+1}})}\Big)^{2\widetilde{s}_{l_{k+1}}(\mathcal{A}_{B},\xi,{\tau(i)})}$$
and for $m_{k}<n<m_{k+1}$, set
$$\mu\big(I_{n}(a_{1},\ldots,a_{n})\big)=\sum\limits_{a_{n+1},\ldots,a_{m_{k+1}}} \mu\big(I_{m_{k+1}}(a_{1},\ldots,a_{n},a_{n+1},\ldots, a_{m_{k+1}})\big).$$
Likewise,  the last summation is taken under the restriction  that $(a_{1},\ldots,a_{m_{k+1}})\in \mathcal{B}_{m_{k+1}}.$

\noindent Step III$\colon$ We have distributed the measure among basic intervals. By the definition of $\widetilde{s}_{l_{k}}(\mathcal{A}_{B},\xi,{\tau(i)})$, we readily check the consistency: for $n\geq 1$ and $(a_{1},\ldots,a_{n})\in \mathcal{B}_{n}$,
$$\mu\big(I_{n}(a_{1},\ldots,a_{n})\big)=\sum\limits_{a_{n+1}}\mu\big(I_{n+1}(a_{1},\ldots,a_{n},a_{n+1})\big).$$
We then extend the measure  to all Borel sets by Kolmogorov extension theorem. The extension measure is also denoted by $\mu$.



\smallskip

From the construction, we know that $\mu$ is supported on $E(B)$ and
$$\mu\big(I_{m_{k}}(a_{1},\ldots,a_{m_{k}})\big)
=\prod\limits_{j=1}^{k}\Big(\frac{1}{q_{l_{j}}(a_{m_{j-1}+1},\ldots,a_{m_{j}})}\Big)^{2\widetilde{s}_{l_{j}}(\mathcal{A}_{B},\xi,{\tau(i)})},~~\sum\limits_{a_{1}\in \mathcal{B}_{1}}\mu\big(I_{1}(a_{1})\big)=1.$$

\subsubsection{\textbf{H\"{o}lder exponent of $\mu$.}} We  shall start with the  study of  a basic interval.

For $0<\varepsilon<{s(\mathcal{A}_{B},\xi,{\tau(i)})}/{4}$,
 by Lemmas \ref{lem29}, \ref{lem41} and the fact that $m_k$ grows exponentially,  we can find  $K\in \mathbb{N}$ such that for any $k, j\geq K,$
\begin{equation}\label{e441}
|\widetilde{s}_{l_{k}}(\mathcal{A}_{B},\xi,{\tau(i)})-s(\mathcal{A}_{B},\xi,{\tau(i)})|<\varepsilon,
\end{equation}
\begin{equation}\label{e442}
|\widetilde{s}_{l_{k}}(\mathcal{A}_{B},\xi,{\tau(i)})-s_j(\mathcal{A}_{B},\xi,{\tau(i)})|<\frac{\varepsilon\log 2}{2\log (B+1)}:=\varepsilon',
\end{equation}
and \begin{equation}\label{e443}
\max\big\{(B+1)^K,2^k\big\}\le \frac14\big(q_{m_k}(a_1,\ldots,a_{m_k})\big)^\varepsilon.
\end{equation}

%

\begin{lem}\label{lem42}
 Let $n\ge m_K$. For  $(a_{1}, \ldots, a_{n})\in \mathcal{B}_{n}$, we have
$$\mu\big(I_{n}(a_{1},\ldots,a_{n})\big)\leq C_0\cdot |I_{n}(a_1,\ldots,a_n)|^{s(\mathcal{A}_{B},\xi,{\tau(i)})-3\varepsilon},$$
where $C_0=(B+1)^{2(l_1+\cdots+l_{K-1})}.$
\end{lem}
\begin{proof}
To shorten notation, we will write $\widetilde{s}_{l_{k}}, s_k$ and $s$ instead of $\widetilde{s}_{l_{k}}(\mathcal{A}_{B},\xi,{\tau(i)}), s_k(\mathcal{A}_{B},\xi,{\tau(i)})$ and $s(\mathcal{A}_{B},\xi,{\tau(i)})$, respectively. Fixing  $(a_{1}, \ldots, a_{n})\in \mathcal{B}_{n}$, we also write $I_n$ for $I_n(a_1,\ldots,a_n)$,
$q_n$ for $q_n(a_1,\ldots,a_n)$ and $q_{l_j}$ for ${q_{l_{j}}(a_{m_{j-1}+1},\ldots,a_{m_{j}})}$ when no confusion can arise. The proof falls naturally into three parts according to the range of $n$.
\smallskip

\underline{\textsc{Case 1$\colon$} $n=m_{k}$ for $k\ge K$}

By Lemmas \ref{lem22} and \ref{lem23}, (\ref{e441}), (\ref{e443}) and the fact $q_{l_j}\le(B+1)^{l_j}$, we obtain
\begin{align*}
\mu(I_{m_{k}})&=\prod\limits_{j=1}^{k} q_{l_j}^{-2\widetilde{s}_{l_j}}=\prod\limits_{j=1}^{K}q_{l_j}^{-2\widetilde{s}_{l_j}}\cdot\prod\limits_{j=K+1}^{k}q_{l_j}^{-2\widetilde{s}_{l_j}}\le C_0\prod\limits_{j=1}^{K}q_{l_j}^{-2(s-\varepsilon)}\cdot\prod\limits_{j=K+1}^{k}q_{l_j}^{-2(s-\varepsilon)}\\
                               &\leq C_{0}2^{2(k-1)}(q_{m_k})^{-2(s-\varepsilon)}\le \frac{C_{0}}4(q_{m_k})^{-2(s-2\varepsilon)}\leq  C_{0}|I_{m_{k}}|^{s-2\varepsilon}.
\end{align*}


\smallskip

\underline{\textsc{Case 2$\colon$} $m_{k}<n< n_{k+1}$  for $k\ge K$}

In this case, we have
$$\mu(I_n)=\sum_{a_{n+1},\ldots,a_{m_{k+1}}}\mu(I_{m_{k+1}})=\sum\prod_{j=1}^{k+1}(q_{l_j})^{-2\widetilde{s}_{l_{j}}}= \prod_{j=1}^{k}(q_{l_j})^{-2\widetilde{s}_{l_{j}}}\cdot \sum (q_{l_{k+1}})^{-2\widetilde{s}_{l_{k+1}}}.$$
We have already seen in \textsc{Case 1} that $\prod_{j=1}^{k}(q_{l_j})^{-2\widetilde{s}_{l_{j}}}\le C_{0}2^{2(k-1)}(q_{m_k})^{-2(s-\varepsilon)}$. And
$$\sum (q_{l_{k+1}})^{-2\widetilde{s}_{l_{k+1}}}\le
 \big({q_{n-m_{k}}(a_{m_{k}+1},\ldots,a_{n})}\big)^{-2(s-\varepsilon)} \cdot
                               \sum \big({q_{m_{k+1}-n}(a_{n+1},\ldots,a_{n_{k+1}},i,\ldots,i)}\big)^{-2\widetilde{s}_{l_{k+1}}}.
 $$
 We then obtain that
 $$ \mu(I_n)\le C_02^{2k}(q_n)^{-2(s-\varepsilon)}\cdot  \sum_{a_{n+1},\ldots,a_{m_{k+1}}} \big({q_{m_{k+1}-n}(a_{n+1},\ldots,a_{n_{k+1}},i,\ldots,i)}\big)^{-2\widetilde{s}_{l_{k+1}}}.$$
Now we need an upper estimate of the last sum. By the definition of $\widetilde{s}_{l_{k+1}}$, we have that
$$ \sum\limits_{1\leq a'_{m_{k}+1},\ldots,a'_n,a_{n+1},\ldots,a_{n_{k+1}}\leq B}
        \big(q_{l_{k+1}}( a'_{m_{k}+1},\ldots,a'_n,a_{n+1},\ldots,a_{n_{k+1}},i,\ldots,i)\big)^{-2\widetilde{s}_{l_{k+1}}}=1.$$
This yields  that
$$\sum
          \big( q_{n-m_{k}}(a'_{m_{k}+1},\ldots,a'_{n})\big)^{-2\widetilde{s}_{l_{k+1}}}\cdot
     \sum
           \big(q_{m_{k+1}-n}(a_{n+1},\ldots,a_{n_{k+1}},i,\ldots,i)\big)^{-2\widetilde{s}_{l_{k+1}}}\le 4.$$
We will bound the first sum  from below to reach the desired upper estimate of the second sum. We consider two cases.

{\rm{(1)}} If $n-m_{k}< K$,
$$\sum_{a'_{m_{k}+1},\ldots,a'_{n}}\big(q_{n-m_k}(a'_{m_{k}+1},\ldots,a'_{n})\big)^{-2\widetilde{s}_{l_{k+1}}}
\geq \big(q_{n-m_k}(B,\ldots,B)\big)^{-2}\geq( {B+1})^{-2K}.$$
And thus, by (\ref{e443}), we reach that
$$
  \mu(I_{n})
\leq C_{0}2^{2k+2}(B+1)^{2K} ({q_{n}} )^{-2(s-\varepsilon)}\le\frac{C_{0}}4(q_{m_k})^{-2(s-3\varepsilon)}
\leq C_{0} |I_{n}|^{s-3\varepsilon}.$$

{\rm{(2)}} If $n-m_{k}\geq K$, then, by (\ref{e442}), (\ref{e443}) and  Remark \ref{rem2}, we have
\begin{align*}
&\sum_{a'_{m_{k}+1},\ldots,a'_{n}}\big(q_{n-m_k}(a'_{m_{k}+1},\ldots,a'_{n})\big)^{-2\widetilde{s}_{l_{k+1}}}
\ge \sum_{a'_{m_{k}+1},\ldots,a'_{n}}\big(q_{n-m_k}(a'_{m_{k}+1},\ldots,a'_{n})\big)^{-2{s}_{n-m_k}-\varepsilon'}\\
&\ge\sum_{a'_{m_{k}+1},\ldots,a'_{n-\lfloor(n-m_k)\xi\rfloor}}\Big({q_{n-m_{k}}(a'_{m_{k}+1},\ldots,a'_{n-\lfloor(n-m_k)\xi\rfloor},
                          i,\ldots,i)}\Big)^{-2s_{n-m_{k}}-\varepsilon'}\\
&\ge\Big({q_{n-m_{k}}(B,\ldots,B)}\Big)^{-\varepsilon'}
\ge({B+1})^{-(n-m_{k})\varepsilon'}
\ge({B+1})^{-n\varepsilon'}
\ge2^{\frac{-n\varepsilon}{2}}.
\end{align*}

Therefore,
$$
    \mu(I_{n})
\leq C_{0}2^{2k+2}2^{\frac{n\varepsilon}{2}} ({q_{n}} )^{-2(s-2\varepsilon)}
\leq \frac{C_{0}}2({q_{n}})^{-2(s-3\varepsilon)}
\leq C_{0}|I_{n}|^{s-3\varepsilon}.
$$

\smallskip

\underline{\textsc{Case 3$\colon$} $n_{k+1}\leq n<m_{k+1}$  for $k\ge K$}

In this case,  since $(a_{n+1},\ldots,a_{m_{k+1}})=(i,\ldots,i),$ we have
$$\mu(I_{n}(a_{1},\ldots,a_{n}))=\mu(I_{m_{k+1}}(a_{1},\ldots,a_{m_{k+1}})),$$
then
$$
   \mu(I_{n})
  \leq C_{0}|I_{m_{k+1}}|^{s-2\varepsilon}
  \leq C_{0}|I_{n}|^{s-2\varepsilon}.
$$
These conclude the verification of Lemma.
\end{proof}

Now we study the H\"older exponent for  the measure of a general  ball $B(x,r)$. 

\begin{lem}\label{lem43} For $x\in E(B)$ and  $r>0$ small enough, we have
$$\mu(B(x,r))\leq C_{0}\cdot r^{s(\mathcal{A}_{B},\xi,{\tau(i)})-4\varepsilon}.$$
\end{lem}

\begin{proof}
Let $x=[a_{1},a_{2},\ldots]$ be its continued fraction expansion.
Let $n\ge K+2$ be the integer such that
$$|I_{n+1}(a_{1},\ldots,a_{n+1})|\leq r<|I_{n}(a_{1},\ldots,a_{n})|.$$
Therefore it follows from Lemmas \ref{lem24} and \ref{lem42} that
\begin{align*}
\mu(B(x,r))&\leq \mu(I_{n-2}(a_{1},\ldots,a_{n-2}))
              \leq C_{0}\cdot|I_{n-2}(a_{1},\ldots,a_{n-2})|^{s(\mathcal{A}_{B},\xi,{\tau(i)})-3\varepsilon}
          \\& \le C_{0}(B+1)^{6}\cdot |I_{n+1}(a_{1},\ldots,a_{n+1})|^{s(\mathcal{A}_{B},\xi,{\tau(i)})-3\varepsilon}
        \\& \le C_{0} \cdot |I_{n+1}(a_{1},\ldots,a_{n+1})|^{s(\mathcal{A}_{B},\xi,{\tau(i)})-4\varepsilon}  \le C_{0}\cdot r^{s(\mathcal{A}_{B},\xi,{\tau(i)})-4\varepsilon}.
\end{align*}
\end{proof}

Applying mass distribution principle (see Lemma \ref{lem213}), letting $\varepsilon\to 0$, we conclude that
$$\dim_{H}E(B)\geq s(\mathcal{A}_{B},\xi,{\tau(i)}).$$

\subsection{Lower bound of $\dim_{H}E(\hat{\nu},\nu)$}\label{S_{4.2}}
We  build a mapping $f$ from $E(B)$ to $ E(\hat\nu,\nu)$ and prove that $f$ is dimension-preserving.

Fix an integer $d>B$. For  $x=[a_{1},a_{2},\ldots]$ in $E(B)$, we remark that the continued fraction of $x$ is the concatenation of $\mathbb B_0=[a_1,\ldots,a_{n_1}]$ and the blocks
$$\mathbb B_k=[\,\underbrace{ i,\ldots, i}_{m_k-n_k}, a_{m_{k}+1}, \ldots, a_{n_{k+1}}] \quad (k=1,2,\ldots).$$
In  the block $\mathbb B_k$, from the beginning we insert a digit $d$ after each $m_k-n_k$ digits to obtain a new block $\mathbb B_k'$, that is,
$$\mathbb B'_k=[\, d, i,\ldots, i, d, a_{m_{k}+1}, \ldots, a_{m_{k}+(m_k-n_k)}, d,
\ldots,   a_{n_{k+1}}].$$
Concatenating the blocks $\mathbb B_0, \mathbb B'_1, \mathbb B'_2,\ldots$, we get
$[\mathbb B_0, \mathbb B'_1, \mathbb B'_2,\ldots]$, which is a continued fraction expansion of some $\bar{x}$. We then define $f(x)=\bar{x}.$
Let $\mathbf{K}=\{k_n\}\subset \mathbb N$ be the collection of the occurrences of  the digit $d$ in the continued expansion of $\bar x$. It is trivially seen that $\mathbf K$ is independent of the choice of $x\in E(B)$, and, in the notation of Lemma \ref{lem214}, $\phi_\mathbf K(\bar x)=x$ for $x\in E(B)$.


Let $h_k$ be the length of the block $\mathbb B'_k.$ Noting that the number of the inserted digit $d$ is at most $\frac{n_{k+1}-m_{k}}{m_{k}-n_{k}}+1=o(h_k)$ in the block $\mathbb B'_k$, we readily check that $\mathbf K$ is a subset of $\mathbb N$ of density zero. Hence
by Lemma \ref{lem214} we have
$$\dim_{H}f\big(E(B)\big)=\dim_{H}E(B).$$
It remains to prove that  $f\big(E(B)\big)$  is a Cantor subset of $E(\hat{\nu},\nu).$

\begin{lem}\label{lem44}
$f\big(E(B)\big)\subset E(\hat{\nu},\nu).$
\end{lem}

\begin{proof}
Fix $\bar{x}\in f\big(E(B)\big)$.

For $\varepsilon>0$ and $n$ large enough, there exists some $k$ such that $(\sum_{j=0}^{k-1}h_j)\leq n<(\sum_{j=0}^{k}h_j)$.
From the construction we deduce  that:
if $n=(\sum_{j=0}^{k-1}h_j)+1$, then
$$|T^{n}(\bar{x})-y|<|I_{m_{k}-n_{k}}(y)|\le|I_{n}(y)|^{\nu-\varepsilon},$$
where the last inequality holds by the fact $\lim_k \frac{\sum_{j=0}^{k-1}h_j}{n_k}=1;$
if $(\sum_{j=0}^{k-1}h_j)< n<(\sum_{j=0}^{k}h_j),$ then
$$|T^{n}(\bar{x})-y|\ge \frac{1}{2(d+2)^{2}}|I_{m_{k}-n_{k}}(y)|\ge|I_n(y)|^{\nu+\varepsilon}.$$
We then prove that  $\nu(\bar{x})=\nu$ by   the arbitrariness of $\varepsilon.$

On the other hand,  for $(\sum_{j=0}^{k-1}h_j)\leq N<(\sum_{j=0}^{k}h_j)$ with $k$ large enough, we pick $n=(\sum_{j=0}^{k-1}h_j)+1$ to obtain that
$$|T^{n}(\bar{x})-y|<|I_{m_{k}-n_{k}}(y)|\le|I_{\sum_{j=0}^{k}h_j}(y)|^{\hat{\nu}-\varepsilon}<|I_{N}(y)|^{\hat{\nu}-\varepsilon}.$$
When $N=(\sum_{j=0}^{k}h_j)$,  for all $n\in [1,N],$ we have that
$$|T^{n}(\bar{x})-y|\ge\frac{1}{2(d+2)^{2}}|I_{m_{k}-n_{k}}(y)|\ge|I_{N}(y)|^{\hat{\nu}+\varepsilon}.$$
We prove that $\hat{\nu}(\bar{x})=\hat{\nu}.$

Hence $\bar{x}\in E(\hat{\nu},\nu),$ as desired.
\end{proof}



Consequently, for $0\le\hat{\nu}\le\frac{\nu}{1+\nu}<\nu<\infty,$ we have
$\dim_{H}E(\hat{\nu},\nu)\geq s(\mathcal{A}_{B},\xi,{\tau(i)})$.
Letting $B\rightarrow\infty$ yileds $\dim_{H}E(\hat{\nu},\nu)\geq s(\xi,{\tau(i)}).$
\bigskip

We conclude this section by determining the  lower bound of $\dim_{H}E(\hat{\nu}, +\infty).$
We first study the case $0<\hat{\nu}<1$. Let
$$ n_{1}=2,~n_{k+1}=n_{k}^{k}+2n_{k},
~m_0=0,~m_{k}=\Big\lfloor\hat{\nu}n_{k}^{k}\Big\rfloor+n_{k},
~B_{k}=\lfloor m_{k}\log m_{k}\rfloor.$$
And thus
$$\lim_{k\to\infty}\frac{m_{k}-n_{k}}{n_{k+1}}=\hat{\nu},
~~\lim_{k\to\infty}\frac{m_{k}-n_{k}}{n_{k}}=\infty,
~~\lim_{k\to\infty}\frac{m_{k}-n_{k}}{m_{k}}=1.$$
Define
\begin{equation*}
  E=\left\{x\in[0,1)\colon
                          a_{n}(x)\le B_k \text{ if } m_{k}<n\le n_{k+1} \text{ for some } k;
                                                           \   a_n(x)=i,  \text{ otherwise}\right\}.
\end{equation*}
As before, for any $x=[a_1,a_2,\ldots]$ in $E,$
we construct an element $\bar{x}:=f(x)\colon$ insert a digit $B_k+1$ after positions $n_{k}$ and $m_{k}+i(m_{k}-n_{k}),$ $0\leq i\leq t_{k}$ in the continued fraction expansion of $x$, where
$t_{k}=\max \{t\in \mathbb{N}\colon m_{k}+t(m_{k}-n_{k})< n_{k+1}\}$; the resulted sequence is the continued fraction of $\bar x$.

The method establishing the the lower bound of $\dim_{H}E(B)$  applies to show that $\dim_{H}E\geq 1/2 .$
Moreover,  $f(E)$ is a subset of $E(\hat{\nu}, +\infty).$ It remains to prove that the Hausdorff dimension of $f(E)$ coincide with the one of $E$. To this end, we shall   show that $f^{-1}$ is a $(1-\varepsilon)$-H\"{o}lder mapping for any $\varepsilon>0$. We remark that Lemma \ref{lem214} may not apply directly here since   $\{B_{k}\}$ is an unbounded sequence.

\begin{lem}
For $\varepsilon>0,$ $f^{-1}$ is a $(1-\varepsilon)$-H\"{o}lder mapping.
\end{lem}

\begin{proof}
We write
$$m_{k}'=m_{k}+\sum_{l=1}^{k-1}(t_{l}+2),~n_{k}'=n_{k}+\sum_{l=1}^{k-1}(t_{l}+2),$$
and define the marked set
$$\mathbf{K}=\big\{m_{k}'+i(m_{k}-n_{k})+1\colon  0\leq i \leq t_{k}, k\geq 1\big\}\cup\big\{ n_{k}'\colon k\geq 1\big\}.$$

Let  $\Delta_{n}=\sharp\{i\leq n\colon i\in \mathbf{K}\}$, where $\sharp$ denotes the cardinality of a finite set. Let  $k\in \mathbb{N}$ such that $m_{k}'\leq n<m_{k+1}'$. We have
\begin{align*}
  \frac{\Delta_{n}\log B_{k}}{n}
&\leq\frac{\Big(\sum_{l=1}^{k-1}(t_{l}+2)+\frac{n-m_{k}'}{m_{k}-n_{k}}+1\Big)\log B_{k}}{n}\\
&\leq \frac{\Big(\sum_{l=1}^{k-1}(t_{l}+2)+1\Big)\log B_{k}}{m_{k}'}+\frac{\log B_{k}}{m_{k}-n_{k}}+\frac{\log B_{k}}{m_{k}}\to 0.
\end{align*}
 So there exists $K\in \mathbb{N}$, such that for $k\geq K$ and $n\geq m_{K}',$
\begin{equation}\label{e48}
(B_{k}+2)^{2\Delta_{n}+4}<2^{(n-1)\varepsilon}.
\end{equation}

For  $\overline{x_{1}}=f(x_{1})$ and $\overline{x_{2}}=f(x_{2})$ in $f(E),$  we assume without loss of generality that
\begin{equation*}
|\overline{x_{1}}-\overline{x_{2}}|<\frac{1}{2(B_{K}+2)^{2}q^{2}_{m_{K}'}(\overline{x_{1}})};
\end{equation*}
otherwise, 
$|f^{-1}(\overline{x_{1}})-f^{-1}(\overline{x_{2}})|<C|\overline{x_{1}}-\overline{x_{2}}|^{1-\varepsilon}$ for some $C,$ as desired.
Let
$$n=\min\{j\geq 1\colon a_{j+1}(\overline{x_{1}})\neq a_{j+1}(\overline{x_{2}})\}.$$
By Lemma \ref{lem22}, we have $m_{k}'\leq n<m_{k+1}'$ for some $k\geq K$ and $n+1<n_{k+1}'.$
Assume  that $\overline{x_{1}}>\overline{x_{2}}$  and $n$ is even  (the same conclusion can be drawn for the remaining cases).
There exist $1\leq \tau_{n+1}(\overline{x_{1}})<\sigma_{n+1}(\overline{x_{2}}) \leq B_{k}+1$ such that
$\overline{x_{1}}\in I_{n+1}(a_{1},\ldots,a_{n},\tau_{n+1}(\overline{x_{1}}))$,
$\overline{x_{2}}\in I_{n+1}(a_{1},\ldots,a_{n},\sigma_{n+1}(\overline{x_{2}})).$
Combining Lemma \ref{lem23} and the construction yields that $\overline{x_{1}}-\overline{x_{2}}$ is greater than the length of basic interval $I_{n+2}(a_{1},\ldots,a_{n},\sigma_{n+1}(\overline{x_{2}}),B_{k}+1).$ This implies that
\begin{align*}
 \overline{x_{1}}-\overline{x_{2}}\geq |I_{n+2}(a_{1},\ldots,a_{n},\sigma_{n+1}(\overline{x_{2}}),B_{k}+1)|
             \geq \frac{1}{2(B_{k}+2)^{4}q^{2}_{n}}.
\end{align*}
Furthermore, noting  that $f^{-1}(\overline{x_{1}}),$ $f^{-1}(\overline{x_{2}})\in I_{n-\Delta_{n}}(c_{1},\ldots,c_{n-\Delta_{n}}),$
where $(c_{1},\ldots,c_{n-\Delta_{n}})$ is obtained by eliminating all the terms $a_i$ with $i\in  \mathbf{K}$ from $(a_1,\ldots,a_n)$, we conclude that
\begin{align*}
 |f^{-1}(\overline{x_{1}})-f^{-1}(\overline{x_{2}})|&\leq |I_{n-\Delta_{n}}(a_{1},\ldots,a_{n-\Delta_{n}})|
\\&\leq \frac{1}{q^{2}_{n-\Delta_{n}}}
  \leq(B_{k}+2)^{2\Delta_{n}}\frac{1}{q^{2}_{n}}
  \leq2|\overline{x_{1}}-\overline{x_{2}}|^{1-\varepsilon},
\end{align*}
where the penultimate inequality follows by (\ref{e48}). This completes the proof.
\end{proof}

Now we deduce that $\dim_{H}E(\hat{\nu},\infty)\geq \frac{1-\varepsilon}{2}$ by Lemma \ref{lem212}.
Letting $\varepsilon\to 0$, we establish that   $\dim_{H}E(\hat{\nu},\infty)\geq \frac{1}{2}$  when $0<\hat\nu<1$. A slight change in the proof actually shows that the estimate $\dim_{H}E(\hat{\nu},\infty)\geq \frac{1}{2}$  also works for $\hat\nu=0$ or $1$.
Indeed, when $\hat{\nu}=0,$ we may take
$$n_{k}=2^{2^{2k}},~ m_{k}=n_{k}^{2}, ~B_{k}=2^{n_{k}};$$
when $\hat{\nu}=1,$ we may take
$$m_{k}=(k+1)!,~ n_{1}=1,~n_{k+1}=m_{k}+\frac{m_{k}}{\log m_{k}}, ~B_{k}=\lfloor2^{\sqrt{m_{k}}}\rfloor.$$



\begin{rem}
Applying the similar arguments as in Sections \ref{S_3} and \ref{S_4}, we  also prove  for any $\hat{\nu}, \nu\ge 0$ that
$$\{x\in[0,1]\colon \hat{\nu}(x)\ge\hat{\nu},~\nu(x)=\nu\}$$
and  $E(\hat{\nu},\nu)$ share the Hausdorff dimension.
\end{rem}

\section{Proofs of Theorems \ref{TZ1} and \ref{TZ3}}
In this section we study the Hausdorff dimensions of the following sets:
$$E(\hat{\nu})=\{x\in[0,1)\colon \hat{\nu}(x)=\hat{\nu}\},$$
$$\mathcal{U}(y,\hat{\nu})=\Big\{x\in[0,1)\colon \forall N\gg1, \exists~ n\in[1,N], \text{ such that } |T^{n}(x)-y|<|I_N(y)|^{\hat{\nu}}\Big\}.$$
A direct corollary of the definition is: if $\hat{\nu}_1>\hat{\nu}\ge 0$,
\begin{equation*}
E(\hat{\nu}_1)\subseteq \mathcal{U}(y,\hat{\nu})\subseteq \{x\in[0,1)\colon \hat{\nu}(x)\ge\hat{\nu}\}.
\end{equation*}
 Hence, the proofs of Theorems \ref{TZ1} and \ref{TZ3} will be divided into two parts:  the upper bound of $\dim_{H}\{x\in[0,1)\colon \hat{\nu}(x)\ge\hat{\nu}\}$ and the lower bound of $\dim_{H}E(\hat{\nu})$.

 Lemma \ref{full} combined with the fact  $E(0,0)\subset E(0)$ implies that the sets $E(0),$ $\mathcal{U}(y,0) $ and $\{x\in[0,1)\colon \hat{\nu}(x)\ge0\}$ are of full Lebesgue measure; we   only need to deal with the case $\hat{\nu}>0.$

We start with the  upper bound of $\dim_{H}\{x\in[0,1)\colon \hat{\nu}(x)\ge\hat{\nu}\}.$
\begin{lem}
If $0<\hat{\nu}\le 1,$ we have $$\dim_{H}\{x\in[0,1)\colon \hat{\nu}(x)\ge\hat{\nu}\}\le s\Big(\frac{4\hat{\nu}}{(1+\hat{\nu})^{2}},{\tau(i)}\Big).$$ If $\hat{\nu}>1,$ then $\dim_{H}\{x\in[0,1)\colon \hat{\nu}(x)\ge\hat{\nu}\}=0.$
\end{lem}

\begin{proof}
For $\varepsilon>0$ small enough, we define
$$E_\varepsilon(\hat{\nu},\nu)=\Big\{x\in[0,1)\colon \hat{\nu}(x)\ge\hat{\nu},~\nu\le\nu(x)\le\frac{\nu+\varepsilon}{1-\varepsilon}\Big\}.$$
Since $$\{x\in[0,1)\colon \hat{\nu}(x)\ge\hat{\nu}\}\subseteq\bigcup_{\nu\in \mathbb{Q}^{+}}E_\varepsilon(\hat{\nu},\nu),$$
where $\mathbb{Q}^{+}$ denotes the set of positive rational numbers,  we have $$\dim_{H}\{x\in[0,1)\colon \hat{\nu}(x)\ge\hat{\nu}\}\le\sup\big\{\dim_{H}E_\varepsilon(\hat{\nu},\nu)\colon \nu\in \mathbb{Q}^{+}\big\}.$$

If $\hat{\nu}>1$, the set $E_\varepsilon(\hat{\nu},\nu)$ is at most countable by Lemmas \ref{lem31} and \ref{lem32}, and  $\dim_{H}\{x\in[0,1)\colon \hat{\nu}(x)\ge\hat{\nu}\}=0.$

If $0<\hat{\nu}\le1,$ we obtain
$$\dim_{H}E_\varepsilon(\hat{\nu},\nu)\leq s\Big(\frac{\nu^{2}}{(\nu-\hat{\nu}+\hat{\nu}\varepsilon+\varepsilon)(1+\nu)},{\tau(i)}\Big)$$
in much the same way as  the proof for the upper bound of $\dim_{H}E(\hat{\nu},\nu)$; we sketch the main differences$\colon$

In Lemma \ref{lem32},  $\limsup\frac{m_{k}-n_{k}}{m_{k}}$ is estimated by
\begin{equation*}
\frac{\nu}{1+\nu}\le\limsup\limits_{k\rightarrow\infty}\frac{m_{k}-n_{k}}{m_{k}}\leq\frac{\nu+\varepsilon}{1+\nu}.
\end{equation*}

The formulae (\ref{e36}) through (\ref{e37}) are replaced by
\begin{align*}
    \sum_{i=1}^{k}(m_{i}-n_{i})
\ge  m_{k}\Big(\frac{(\nu+\varepsilon)^{2}}{(\nu-\hat{\nu}+\hat{\nu}\varepsilon+\varepsilon)(1+\nu)}-\varepsilon\Big),
\end{align*}
$$1+\nu\le\limsup_{k\to\infty}\frac{m_k}{n_k}\le\frac{1+\nu}{1-\varepsilon},$$
$$\liminf_{k\to\infty}\frac{m_k}{n_{k+1}}\ge\frac{\hat{\nu}(1+\nu)}{\nu+\varepsilon},$$
$$\liminf_{k\to\infty}\frac{\sum_{i=1}^{k}(m_i-n_i)}{m_{k+1}}
\ge \frac{\hat{\nu}(\nu+\varepsilon)(1-\varepsilon)}{(\nu-\hat{\nu}+\hat{\nu}\varepsilon+\varepsilon)(1+\nu)},$$
resepctively. The set $\Omega$ is replaced by
\begin{align*}
  \Big\{(\{n_{k}\},\{m_{k}\})\colon \liminf_{k\to\infty}\frac{m_{k}-n_{k}}{n_{k+1}}\ge\hat{\nu},&~\nu\leq\limsup_{k\to\infty}
                                     \frac{m_{k}-n_{k}}{n_{k}}\leq\frac{\nu+\varepsilon}{1-\varepsilon},\\&
                                       n_k<m_k<n_{k+1} ~\text {for all} ~k\geq 1 \Big\}.
\end{align*}

Finally, since the function $\frac{\nu^{2}}{(\nu-\hat{\nu}+\hat{\nu}\varepsilon+\varepsilon)(1+\nu)}$ of $\nu$ attains its minimum at   $\nu=\frac{2\hat{\nu}-2(\hat{\nu}+1)\varepsilon}{1-\hat{\nu}+(\hat{\nu}+1)\varepsilon}$, we have by Lemma \ref{lem210}(2) that
\begin{align*}
  \dim_{H}\{x\in[0,1)\colon \hat{\nu}(x)\ge\hat{\nu}\}
  &\leq \sup \left\{s\Big(\frac{\nu^{2}}{(\nu-\hat{\nu}+\hat{\nu}\varepsilon+\varepsilon)(1+\nu)},{\tau(i)}\Big)\colon \nu \in\mathbb{Q}^{+} \right\}\\
  &\leq s\Big(\frac{4\big(\hat{\nu}-(\hat{\nu}+1)\varepsilon\big)}{\big(1+\hat{\nu}-(\hat{\nu}+1)\varepsilon\big)^{2}},{\tau(i)}\Big)
  \rightarrow s\Big(\frac{4\hat{\nu}}{(1+\hat{\nu})^{2}},{\tau(i)}\Big)
\end{align*}
as $\varepsilon\to 0$.
\end{proof}

We now deal with the lower bound of the $\dim_{H}E(\hat{\nu})$ for $0<\hat{\nu}\le1.$

\begin{lem}
For $0<\hat{\nu}\le1,$ we have $\dim_{H}E(\hat{\nu})\ge s\Big(\frac{4\hat{\nu}}{(1+\hat{\nu})^{2}},{\tau(i)}\Big).$
\end{lem}
\begin{proof}
Noting that $E(\hat{\nu},\nu)$ is a subset of $E(\hat{\nu})$ for $\nu\ge\frac{\hat{\nu}}{1-\hat{\nu}}$ (or $\hat{\nu}\le\frac{\nu}{1+\nu}$), we have
$$\dim_{H}E(\hat{\nu})\geq \dim_{H}E(\hat{\nu},\nu)=s\Big(\frac{\nu^{2}}{(1+\nu)(\nu-\hat{\nu})},{\tau(i)}\Big).$$
Since the function $\frac{\nu^{2}}{(1+\nu)(\nu-\hat{\nu})}$ is continuous for $\nu\in [\frac{\hat{\nu}}{1-\hat{\nu}},\infty],$
and attains its minimum at   $\nu=\frac{2\hat{\nu}}{1-\hat{\nu}},$ so by Lemma \ref{lem210}(2), we have
$$\dim_{H}E(\hat{\nu})\geq\dim_{H}E\Big(\hat{\nu},\frac{2\hat{\nu}}{1-\hat{\nu}}\Big)=s\Big(\frac{4\hat{\nu}}{(1+\hat{\nu})^{2}},{\tau(i)}\Big).$$
\end{proof}

\section{Proof of Theorem \ref{thm1}}

In this section, we prove  Theorem \ref{thm1} by considering the upper and lower bounds of $\dim_{H}\big(F(\alpha)\cap G(\beta)\big)$ respectively. Recall that
\begin{equation*}
F(\alpha)=\left\{x\in[0,1)\colon \liminf_{n\to\infty}\frac{R_{n}(x)}{n}=\alpha\right\},
G(\beta)=\left\{x\in[0,1)\colon \limsup_{n \to\infty}\frac{R_{n}(x)}{n}=\beta\right\}.
\end{equation*}
The proof of Theorem \ref{thm1} goes along the  lines as that of Theorem \ref{TZ2} with some minor modifications.

Noting that $\Big\{x\in [0,1)\colon \lim_{n\to\infty}\frac{R_{n}(x)}{\log_{\frac{\sqrt{5}+1}{2}}n}=\frac{1}{2}\Big\}\subset F(0)\cap G(0),$
 we have $F(0)\cap G(0)$ is of full Lebesgue measure.  Furthermore, since $F(\alpha)\cap G(1)\subset G(1)$ and $F(0)\cap G(\beta)\subset G(\beta)$, we have
$\dim_{H}F(\alpha)\cap G(1)\leq s(1,\tau(1))$ and $\dim_{H}\big(F(0)\cap G(\beta)\big)\le s(\beta,\tau(1))$.
We only need to consider the case $0<\alpha\le\beta<1.$
\subsection{Upper bound of $\dim_{H}\big(F(\alpha)\cap G(\beta)\big)$}

For $x=[a_{1}(x),a_{2}(x),\ldots]\in F(\alpha)\cap G(\beta)$ with non-periodic continued fraction expansion, we associate $x$ with two sequences $\{n_k\}$ and $\{m_k\}$ that satisfy the following properties$\colon$

{\rm{(1)}} $n_k<m_k<n_{k+1}<m_{k+1}$ for $k\ge1$;

{\rm{(2)}} $a_{n_k}(x)=\cdots=a_{m_k}(x)$ for $k\ge1;$

{\rm{(3)}} $\displaystyle\liminf_{k\to\infty}\frac{m_k-n_k}{n_{k+1}}=\frac{\alpha}{1-\alpha},$ $\displaystyle\limsup_{k\to\infty}\frac{m_k-n_k}{m_k}=\beta;$

{\rm{(4)}} The sequence $\{m_k\}$ grows exponentially;

{\rm{(5)}} Write $\xi=\frac{\beta^{2}(1-\alpha)}{\beta-\alpha}$. For any $\varepsilon>0$, there exist infinitely many $k$ such that
 \begin{equation*}
 \sum_{i=1}^{k}(m_{i}-n_{i}+1)\ge   m_{k}(\xi-\varepsilon).
  \end{equation*}
To this end, we define two ascending sequences $\{n_{k}'\}$ and $\{m_{k}'\}$ as follows$\colon$
$$n_{1}'=1,~ a_{n_{k}'}(x)= a_{n_{k}'+1}(x)= \cdots =a_{m_{k}'}(x)\neq a_{m_{k}'+1}(x),~ n_{k+1}'=m_{k}'+1.$$
Since $\beta=\limsup R_n(x)/n>0$, we have that $\limsup_{k\to\infty}(m_{k}'-n_{k}')= +\infty$, which  enables us to pick a non-decreasing subsequence of $\{(n_k', m_k')\}_{k\geq 1}\colon$ put $(n_{1},m_{1})=(n'_{1},m'_{1});$
having choosen $(n_{k}, m_k)=(n_{j_{k}}', m_{j_{k}}')$  for $k\ge1$, we set
$j_{k+1}=\min\big\{j>j_k\colon m_{j}'-n_{j}'> m_{k}-n_{k}\big\},$
and put $(n_{k+1},m_{k+1})=(n_{j_{k+1}}', m_{j_{k+1}}')$.
\smallskip

We readily check the following properties$\colon$
\smallskip

\textbf{(a)} the sequence $\{m_{k}-n_{k}\}_{k\geq 1}$ is non-decreasing and $\lim_{k\to\infty} (m_{k}-n_{k})= +\infty.$
\smallskip

\textbf{(b)} If $m_{k}\leq n\leq n_{k+1}+(m_{k}-n_{k})$ for $k\geq 1$, then $R_{n}(x)=m_{k}-n_{k}+1$, and
\begin{equation*}
\frac{m_{k}-n_{k}+1}{n_{k+1}+(m_{k}-n_{k})}\leq \frac{R_{n}(x)}{n}\leq \frac{m_{k}-n_{k}+1}{m_{k}}.
\end{equation*}

\textbf{(c)} If $n_{k+1}+(m_{k}-n_{k})<n< m_{k+1}$ for   $k\geq 1$, then $R_{n}(x)=n-n_{k+1}+1$, and
\begin{equation*}
\frac{m_{k}-n_{k}+1}{n_{k+1}+(m_{k}-n_{k})}\leq \frac{R_{n}(x)}{n}\leq \frac{m_{k+1}-n_{k+1}+1}{m_{k+1}}.
\end{equation*}

Properties \textbf{(a)} and \textbf{(b)} imply
\begin{equation}\label{e63}
\alpha=\liminf_{n\to\infty}\frac{R_{n}(x)}{n}
      =\liminf_{k\rightarrow\infty}\frac{m_{k}-n_{k}+1}{n_{k+1}+(m_{k}-n_{k})}
\end{equation}
and
\begin{equation}\label{e64}
\beta=\limsup_{n\to\infty}\frac{R_{n}(x)}{n}
     =\limsup_{k\rightarrow\infty}\frac{m_{k+1}-n_{k+1}+1}{m_{k+1}}.
\end{equation}
From  (\ref{e63})  we obtain that
\begin{equation}\label{e65}
\liminf_{k\to\infty}\frac{R_{n_{k+1}}(x)}{n_{k+1}}=\liminf_{k\to\infty}\frac{m_{k}-n_{k}+1}{n_{k+1}}
                                                  =\frac{\alpha}{1-\alpha},
\end{equation}
which combined with (\ref{e64}) yields $\frac\alpha{1-\alpha}\le\beta$, or equivalently  $\alpha\leq \frac{\beta}{1+\beta}$.
And thus the set $F(\alpha)\cap G(\beta)$ is at most countable when $\alpha>\frac{\beta}{1+\beta}$. Moreover, the equality (\ref{e65})   implies that $\{m_{k}\}_{k\geq 1}$ grows at least exponentially,
namely, there exists $C>0$, independent of $x,$ such that $k \leq C\log{m_{k}}$ for $k$ large enough. Further, by (\ref{e64}) and (\ref{e65}), we also have
\begin{equation*}
\liminf_{k\to\infty}\frac{m_{k}}{n_{k+1}}
\ge \liminf_{k\to\infty}\frac{m_{k}-n_{k}}{n_{k+1}}
\cdot\liminf_{k\to\infty}\frac{m_{k}}{m_{k}-n_{k}}
=\frac{\alpha}{\beta(1-\alpha)},
\end{equation*}
which combined with  the Stolz-Ces\`aro theorem implies that
\begin{equation*}
\begin{split}
  \liminf_k\frac{\sum_{i=1}^{k-1}(m_i-n_i+1)}{m_{k}}
\ge \frac{\alpha\beta(1-\beta)}{\beta-\alpha},
\end{split}
\end{equation*}
and thus, for $\varepsilon>0$ and $k$ large enough, 
\begin{equation*}
\sum_{i=1}^{k}(m_{i}-n_{i}+1)\geq \Big(\frac{\alpha\beta(1-\beta)}{\beta-\alpha}-\frac{\varepsilon}{2}\Big)m_{k}+(m_{k}-n_{k}+1).
\end{equation*}
Since there exist infinitely many $k$ such that
\begin{equation}\label{e610}
    m_{k}-n_{k}+1\ge
     m_{k}(\beta-\frac{\varepsilon}{2}),
\end{equation}
  Property (5) holds for such $k$.

\medskip

\underline{Covering of $F(\alpha)\cap G(\beta)$.}\quad We collect all sequences  $(\{n_{k}\}, \{m_{k}\})$  associated with $x\in F(\alpha)\cap G(\beta)$ as above to form a set
$$\Omega=\Big\{(\{n_{k}\},\{m_{k}\})\colon \text{Properties } (1) ~\& ~(3) \text{ are  fulfilled}\Big\}.$$
For $(\{n_{k}\},\{m_{k}\})\in \Omega$ and $\{b_{k}\}\subset \mathbb N$, write
\smallskip
\begin{align*}
&H(\{n_{k}\},\{m_{k}\})=\big\{x\in[0,1)\colon \text{Property } (2) \text{ is fulfilled}\big\},\\&
\Lambda_{k, m_k}=\Big\{(n_1,m_1; \ldots; n_{k-1},m_{k-1};n_{k})\colon n_{1}<m_{1}<\cdots<m_{k-1}<n_{k}, (\ref{e610})\text{ holds} \Big\},\\ &
\mathcal{D}_{n_1,m_1;\ldots;n_{k}, m_{k}}(\{b_{k}\})=\big\{(\sigma_{1},\ldots,\sigma_{m_{k}})\in \mathbb{N}^{m_{k}}\colon \sigma_{n_{j}}=\cdots=\sigma_{m_{j}}=b_{j}\text{ for all }1\leq j\leq k \big\}.
\end{align*}
We obtain a covering of $ F(\alpha)\cap G(\beta)$:
\begin{equation*}
\begin{split}
  &F(\alpha)\cap G(\beta) \subseteq \bigcup_{(\{n_{k}\},\{m_{k}\})\in \Omega}H(\{n_{k}\},\{m_{k}\}) \\
                   &\subseteq\bigcap_{K=1}^{\infty} \bigcup_{k=K}^{\infty}\bigcup_{m_{k}= e^{\frac{k}{C}}}^{\infty} \bigcup_{(n_{1},m_{1},\ldots,m_{k-1},n_{k})\in\Lambda_{k, m_k}}
                             \bigcup_{(b_{1},\ldots,b_{k})\in \mathbb{N}^{k}}
                             \bigcup_{(a_{1},\ldots,a_{m_{k}})\in \mathcal{D}_{n_1,m_1;\ldots;n_{k}, m_{k}}(\{b_{k}\})}I_{m_{k}}(a_{1},\ldots,a_{m_{k}}).
\end{split}
\end{equation*}
Writing $t=s(\xi-\varepsilon,{\tau(1)})+\frac{\varepsilon}{2}$, we estimate the $(t+\frac{\varepsilon}{2})$-dimensional Hausdorff measure of $F(\alpha)\cap G(\beta)$. Setting  $\psi(m_{k})=m_{k}-\sum_{i=1}^{k}(m_{i}-n_{i}+1),$ $M=512\sum_{i=1}^{\infty}\big(\tau(i)\big)^{-2t}.$ For sufficiently large $k$, we first have the following estimate$\colon$
\begin{equation}\label{e616}
\begin{split}
& \sum\limits_{b_{k}=1}^{\infty}\cdots\sum\limits_{b_{1}=1}^{\infty}\sum\limits_{(a_{1},\ldots,a_{m_{k}})\in \mathcal{D}_{n_{k}, m_{k}}(\{b_{k}\})}|I_{m_{k}}(a_{1},\ldots,a_{m_{k}})|^{t+\frac{\varepsilon}{2}}\\
\le&\sum\limits_{b_{k}=1}^{\infty}\cdots\sum\limits_{b_{1}=1}^{\infty}\sum\limits_{a_{1},\ldots,a_{\psi(m_{k})}\in\mathbb{N}}4^{k(t+\frac{\varepsilon}{2})}\Big(\frac{1}{q_{\psi(m_{k})}(a_{1},\ldots,a_{\psi(m_{k})})}\Big)^{2t+\varepsilon}\prod_{j=1}^{k}\Big(\frac{1}{q_{m_j-n_j+1}(b_j,\ldots,b_j)}\Big)^{2t+\varepsilon}\\
\le&4^{k(t+\frac{\varepsilon}{2})}\sum\limits_{a_{1},\ldots,a_{\psi(m_{k})}\in\mathbb{N}}\Big(\frac{1}{q_{\psi(m_{k})}(a_{1},\ldots,a_{\psi(m_{k})})}\Big)^{2t+\varepsilon}\prod_{j=1}^{k}\left(\sum_{i=1}^{\infty}\Big(\frac{1}{q_{m_j-n_j+1}(i,\ldots,i)}\Big)^{2t+\varepsilon}\right)\\
\le&(4^{t+\frac{\varepsilon}{2}}M)^{k}\sum\limits_{a_{1},\ldots,a_{\psi(m_{k})}\in\mathbb{N}}\Big(\frac{1}{q_{\psi(m_{k})}(a_{1},\ldots,a_{\psi(m_{k})})}\Big)^{2t+\varepsilon}\prod_{j=1}^{k}\Big(\frac{1}{q_{m_j-n_j+1}(1,\ldots,1)}\Big)^{2t+\varepsilon}\\
\le&(16^{t+\frac{\varepsilon}{2}}M)^{k}\sum\limits_{a_{1},\ldots,a_{\psi(m_{k})}\in \mathbb{N}}\Big(\frac{1}{q_{m_{k}}(a_{1},\ldots,a_{\psi(m_{k})},1,\ldots,1)}\Big)^{2t+\varepsilon}\\
\le&(16^{t+\frac{\varepsilon}{2}}M)^{k}\sum \limits_{a_{1},\ldots,a_{m_{k}-\lfloor m_{k}(\xi-\varepsilon)\rfloor}\in \mathbb{N}}\Big(\frac{1}{q_{m_{k}}(a_{1},\ldots,a_{m_{k}-\lfloor m_{k}(\xi-\delta)\rfloor},1,\ldots,1)}\Big)^{2s_{m_{k}}(\xi-\varepsilon,{\tau(1)})+\varepsilon}\\
\le&(16^{t+\frac{\varepsilon}{2}}M)^{k}\Big(\frac{1}{2}\Big)^{\frac{m_{k}-1}{2}\varepsilon},
\end{split}
\end{equation}
where the third inequality holds since
\begin{equation*}
\begin{split}
    \sum\limits_{i=1}^{\infty} \Big(\frac{1}{q_{n}(i,\ldots,i)}\Big)^{2t+\varepsilon}
   &=\Big(\frac{1}{q_{n}(1,\ldots,1)}\Big)^{2t+\varepsilon} \sum\limits_{i=1}^{\infty}\Big(\frac{q_{n}(1,\ldots,1)}{q_{n}(i,\ldots,i)}\Big)^{2t+\varepsilon}\\
   &\le\Big(\frac{1}{q_{n}(1,\ldots,1)}\Big)^{2t+\varepsilon} \sum\limits_{i=1}^{\infty}\Big(\frac{4\tau(1)}{\tau(i)}\Big)^{2t+\varepsilon};
\end{split}
\end{equation*}
the penultimate one follows from (\ref{e610}) and (\ref{e37}) and the last one is by Remark \ref{rem2} and Lemma \ref{lem21}(1).

Hence,
\begin{equation*}
\begin{split}
\mathcal{H}^{t+\frac{\varepsilon}{2}}\big(F(\alpha)\cap G(\beta)\big)
\leq &\liminf_{K\to\infty}\sum\limits_{k=K}^{\infty}\sum_{m_{k}= e^{\frac{k}{C}}}^{\infty}
                                                    \sum_{(n_{1},m_{1},\ldots,m_{k-1},n_{k})\in\Lambda_{k, m_k}}\sum\limits_{b_{k}=1}^{\infty}\cdots
                                                    \sum\limits_{b_{1}=1}^{\infty}\\
     &\times\sum\limits _{(a_{1},\ldots,a_{m_{k}})\in\mathcal{D}_{n_{k}, m_{k}}(\{b_{k}\})}|I_{m_{k}}(a_{1},\ldots,a_{m_{k}})|^{t+\frac{\varepsilon}{2}}\\
\overset{\text{(\ref{e616})}}{\leq}&\liminf_{K\to\infty}\sum\limits_{k=K}^{\infty}\sum\limits_{m_{k}=e^{\frac{k}{C}}}^{\infty}
                                                    \sum_{n_{k}=1}^{m_{k}}\sum_{m_{k-1}=1}^{n_{k}}\cdots\sum\limits_{m_{1}=1}^{n_{2}}
                                                    \sum\limits_{n_{1}=1}^{m_{1}}(16^{t+\frac{\varepsilon}{2}}M)^{k}\Big(\frac{1}{2}\Big)^{\frac{m_{k}-1}{2}\varepsilon}\\
\leq&\liminf_{K\to\infty}\sum\limits_{k=K}^{\infty} \sum\limits_{m_{k}=e^{\frac{k}{C}}}^{\infty}(16^{t+\frac{\varepsilon}{2}}Mm_{k})^{2C\log m_{k}}\Big(\frac{1}{2}\Big)^{\frac{m_{k}-1}{2}\varepsilon}\\ \leq&\liminf_{K\to\infty}\sum\limits_{k=K}^{\infty} \sum\limits_{m_{k}=e^{\frac{k}{C}}}^{\infty}\Big(\frac{1}{2}\Big)^{\frac{m_{k}-1}{4}\varepsilon}
\leq \frac{1}{1-(\frac{1}{2})^{\frac{\varepsilon}{4}}}\sum\limits_{k=1}^{\infty} \Big(\frac{1}{2^{\varepsilon}}\Big)^{\frac{e^{\frac{k}{C}}-1}{4}} < +\infty,
\end{split}
\end{equation*}
where the penultimate one holds since $(16^{t+\frac{\varepsilon}{2}}Mk)^{2C\log k}<2^{\frac{k-1}{4}\varepsilon}$ for $k$ large enough.

\subsection{Lower bound of $\dim_{H}\big(F(\alpha)\cap G(\beta)\big)$}
Note that $F(\alpha)\cap F(\beta)$ is at most countable for $\alpha>\frac{\beta}{1+\beta}$; we assume that $\alpha\le \frac{\beta}{1+\beta}.$
Let $\{n_{k}\}$ and $\{m_{k}\}$ be two strictly increasing sequences   satisfying the following conditions$\colon$

{\rm{(1)}} $(m_{k}-n_{k})\leq(m_{k+1}-n_{k+1})$ and $n_{k}< m_{k}< n_{k+1}$ for $k\geq 1$;

{\rm{(2)}} $\lim_{k\to\infty}\frac{m_{k}-n_{k}}{n_{k+1}}=\frac{\alpha}{1-\alpha}$;

{\rm{(3)}} $\lim_{k\to\infty}\frac{m_{k}-n_{k}}{m_{k}}=\beta$.

With the help of these sequences, we construct a Cantor subset of $F(\alpha)\cap G(\beta)$  to provide a lower bound estimate of its dimension.  The set $E(B)$ is defined in much the same way as in Section \ref{S_{4.2}}, the only difference being that the digit  $i$ is replaced by the digit $1$;   the mapping $f$ is defined in exact the same way. It remains to  verify that the set $f\big(E(B)\big)$ is a subset of $F(\alpha)\cap G(\beta).$

\begin{lem}
For any $B\ge 2,$ $f\big(E(B)\big)\subset F(\alpha)\cap G(\beta).$
\end{lem}

\begin{proof}
Recall the definitions of
$$t_{k}=\max \{t\in \mathbb{N}\colon m_{k}+t(m_{k}-n_{k})< n_{k+1}\},$$
$$m_{k}'=m_{k}+\sum_{l=1}^{k-1}(t_{l}+2),~n_{k}'=n_{k}+\sum_{l=1}^{k-1}(t_{l}+2).$$
We know that
$$\lim_{k\to\infty}\frac{\sum_{l=1}^{k}(t_{l}+2)}{n_{k+1}}=0,\quad \lim_{k\to\infty}\frac{n_k}{n_{k}'}=\lim_{k\to\infty}\frac{m_k}{m_k'}=1.$$

For  $x\in f(E(B))$, and $m_{k}'\leq n<m_{k+1}'$ with $k\in \mathbb{N},$ we have that
\begin{equation*}
 R_{n}(x)=\left\{\begin{array}{ll}
                                 m_{k}-n_{k},& \ \text{if}~ m_{k}'\leq n\leq n_{k+1}'+m_{k}-n_{k}, \\
                                 n-n_{k+1}', & \ \text{if}~n_{k+1}'+m_{k}-n_{k}< n< m_{k+1}' .\end{array}\right.
\end{equation*}
Observing that for $n_{k+1}'+m_{k}-n_{k}< n< m_{k+1}',$
\begin{equation*}
\frac{m_{k}-n_{k}}{n_{k+1}'+m_{k}-n_{k}}\leq\frac{R_{n}(x)}{n}=\frac{n-n_{k+1}'}{n}\leq\frac{m_{k+1}-n_{k+1}}{m_{k+1}'},
\end{equation*}
we deduce that
\begin{align*}
  \liminf_{n\to\infty}\frac{R_{n}(x)}{n}
=\lim_{k\to\infty}\frac{R_{n_{k+1}'+m_{k}-n_{k}}(x)}{n_{k+1}'+m_{k}-n_{k}}
=\lim_{k\to\infty}\frac{m_{k}-n_{k}}{n_{k+1}+m_{k}-n_{k}}
=\alpha,
\end{align*}
and
\begin{align*}
  \limsup_{n\to\infty}\frac{R_{n}(x)}{n}
= \lim_{k\to\infty}\frac{m_{k+1}-n_{k+1}}{m_{k+1}'}
= \lim_{k\to\infty}\frac{m_{k+1}-n_{k+1}}{m_{k+1}}
=\beta.
\end{align*}
Hence $x\in F(\alpha)\cap G(\beta)$.
\end{proof}

\section{Proof of Theorem \ref{thm2}}

The proof of Theorem \ref{thm2} will be divided into two parts according as $\alpha=0$ or $0<\alpha\le1.$
We first note that $F(\alpha)\cap G(\beta)$ is at most countable for $\alpha>\frac{1}{2}\ge\frac{\beta}{1+\beta}$ by (\ref{e64}) and (\ref{e65}).
Moreover, since $G(0)\subseteq F(0)$ and $G(0)$ is of full Lebesgue measure, we have $\dim_H F(0)=1$.
Hence, we only need to deal with  the case $0<\alpha\le\frac{1}{2}.$

\textbf{Lower bound of $F(\alpha)$.} Since $F(\alpha)\cap G(\beta)\subseteq F(\alpha)$ for any $\beta\geq \frac{\alpha}{1-\alpha},$  we have
$$\dim_{H}F(\alpha)\geq \dim_{H}F(\alpha)\cap G(\beta)=s\Big(\frac{\beta^{2}(1-\alpha)}{\beta-\alpha},{\tau(1)}\Big).$$
The function $\frac{\beta^{2}(1-\alpha)}{\beta-\alpha}$ is continuous for $\beta\in [\frac{\alpha}{1-\alpha},1],$
and attains its minimum at the point $\beta=2\alpha\ge \frac{\alpha}{1-\alpha},$ so by Lemma \ref{lem210}(2), we obtain
$$\dim_{H}F(\alpha)\geq\dim_{H}\big(F(\alpha)\cap G(2\alpha)\big)=s\big(4\alpha(1-\alpha),{\tau(1)}\big).$$

\textbf{Upper bound of $F(\alpha)$.} For $x\in F(\alpha),$ there exists $\beta_{0}\in(0,1]$ such that
$$\liminf_{n\to\infty}\frac{R_{n}(x)}{n}=\alpha,~\limsup_{n\to\infty}\frac{R_{n}(x)}{n}=\beta_{0}.$$
Then $0<\alpha\leq\frac{\beta_{0}}{1+\beta_{0}}<\beta_{0}< 1$ or $\beta_{0}=1$.
If $0<\alpha\leq\frac{\beta_{0}}{1+\beta_{0}}<\beta_{0}< 1$, then for $0<\varepsilon<\frac{\alpha(1-2\alpha)}{2(1-\alpha)},$
there exists $\beta\in \mathbb{Q}^{+}$ such that $0<\alpha\le\frac{\beta_0}{1+\beta_0}\le\beta\leq\beta_{0}\leq \beta+\varepsilon< 1.$

Let
$$E_{\alpha,\beta,\epsilon}=\Big\{x\in [0,1)\colon \liminf_{n\to\infty}\frac{R_{n}(x)}{n}=\alpha,
                                                   ~\beta \leq \limsup_{n\to\infty}\frac{R_{n}(x)}{n} \leq \beta + \varepsilon < 1\Big\},$$
we have
$$F(\alpha)\subseteq\Big(\bigcup_{\beta\in \mathbb{Q}^{+}}E_{\alpha,\beta,\varepsilon}\Big)\cup \big(F(\alpha)\cap G(1)\big).$$
So $\dim_{H}F(\alpha)\leq \max\Big\{\frac{1}{2}, \sup\big\{\dim_{H}E_{\alpha,\beta,\varepsilon}\colon \beta\in \mathbb{Q}^{+}\big\}\Big\}.$

It remains to estimate the upper bound of  $\dim_{H}E_{\alpha,\beta,\varepsilon}.$
Following the same line as the proof for the upper bound of $\dim_{H}\big(F(\alpha)\cap G(\beta)\big),$ we obtain that
$$\dim_{H}E_{\alpha,\beta,\varepsilon}\leq s\Big(\frac{\beta^{2}(1-\alpha)}{\beta-\alpha+\varepsilon},{\tau(1)}\Big).$$
Thus, since the function
            $\frac{\beta^{2}(1-\alpha)}{\beta-\alpha+\epsilon}$
with respect to $\beta$  attains its  minimum at   $\beta=2(\alpha-\varepsilon)$, we have that
\begin{align*}
  \dim_{H}F(\alpha)&\leq \sup \left\{s\Big(\frac{\beta^{2}(1-\alpha)}{\beta-\alpha+\varepsilon}\Big)\colon\beta \in \mathbb{Q}^{+} \right\}
\leq s\big(4(\alpha-\varepsilon)(1-\alpha),{\tau(1)}\big)\\&
\rightarrow s\big(4\alpha(1-\alpha),{\tau(1)}\big)\quad \text{as }\varepsilon\to 0.
\end{align*}

\medskip

{\noindent \bf  Acknowledgements}. This work is supported by NSFC No. 12171172, 12201476. The authors would like to express their gratitude to Professors  Bao-Wei Wang and Jian Xu for helpful discussions during the preparation of the paper.

\end{document}